\documentclass[11pt]{amsart}
\usepackage{verbatim}

\usepackage[pdftex,
colorlinks, 
bookmarksnumbered, 
bookmarksopen, 
linkcolor=blue, 
citecolor=purple, 
urlcolor=orange, 
]{hyperref}

\usepackage[latin1]{inputenc}

\usepackage{amssymb}
\usepackage{amsthm}
\usepackage{amsfonts}
\usepackage{latexsym}
\usepackage{amsmath}
\usepackage{mathrsfs}
\usepackage{xfrac}
\usepackage{mathabx}
\usepackage{calligra}

\usepackage{graphicx}
\usepackage{subfigure}
\usepackage[justification=centering]{caption}

\usepackage{xcolor}

\usepackage{url}

\newcommand\bcdot{\ensuremath{%
  \mathchoice%
   {\mskip\thinmuskip\lower0.2ex\hbox{\scalebox{1.5}{$\cdot$}}\mskip\thinmuskip}}%
   {\mskip\thinmuskip\lower0.2ex\hbox{\scalebox{1.5}{$\cdot$}}\mskip\thinmuskip}%
   {\lower0.3ex\hbox{\scalebox{1.2}{$\cdot$}}}%
   {\lower0.3ex\hbox{\scalebox{1.2}{$\cdot$}}}%
}

\newtheorem{thm}{Theorem}\numberwithin{thm}{section}

\newtheorem{prop}[thm]{Proposition}
\newtheorem{lem}[thm]{Lemma}
\newtheorem{cor}[thm]{Corollary}

\theoremstyle{definition}

\newtheorem{rem}[thm]{Remark}

\newtheorem{dfn}[thm]{Definition}

\newcommand{\spa}{\hspace{7pt}}
\newcommand{\dd}[3]{\frac{d^{#1}#2}{d#3^{#1}}}
\newcommand{\pa}[3]{\frac{\partial^{#1} #2}{\partial #3^{#1}}}

\newcommand{\R}{\mathbb{R}}
\newcommand{\N}{\mathbb{N}}
\newcommand{\CF}{\mathbb{C}}

\newcommand{\e}{\varepsilon}
\renewcommand{\l}{\lambda}
\newcommand{\pr}[1]{#1^\prime}
\newcommand{\Norm}[2]{\left\| #2\right\|_{#1}}
\renewcommand{\L}[2]{L^{#1}(#2)}
\newcommand{\LW}[3]{L^{#1}_{#3}(#2)}

\newcommand{\Sum}[3]{\sum_{#1 = #2}^{#3}}

\renewcommand{\H}{\mathbf{H}_n}
\newcommand{\Z}{\mathbb{Z}}

\newcommand{\bracket}[2]{\left\langle #1 , #2\right\rangle}

\newcommand{\nn}{\nonumber}
\newcommand{\h}{\mathfrak{h}_n}

\newcommand{\Abs}[1]{\left| #1\right|}
\newcommand{\subbracket}[3]{{\left\langle #2, #3\right\rangle}_{ #1}}

\newcommand{\supp}{\mathop{\mathrm{supp}}}

\newcommand{\F}{\mathcal{F}}

\newcommand{\SF}[1]{\mathscr{S}(\mathbb{R}^{#1})}

\newcommand{\SD}[1]{\mathscr{S}^{\pr{}}(\mathbb{R}^{#1})}

\newcommand{\SFG}[1]{\mathscr{S}(#1)}
\newcommand{\SDG}[1]{\mathscr{S}^{\pr{}}(#1)}
\newcommand{\HG}[2]{\mathbf{H}_{#1, #2}}
\newcommand{\HA}[2]{\mathfrak{h}_{#1, #2}}

\newcommand{\HP}{\udot}
\newcommand{\Lie}[1]{\frak{#1}}

\newcommand{\Ad}{\mathop{\mathrm{Ad}}}
\newcommand{\coAd}{\mathop{\mathrm{Ad}^*}}

\newcommand{\coad}{\mathop{\mathrm{ad}^*}}

\newcommand{\MOP}[1]{\mathop{\mathrm{#1}}}

\newcommand{\I}{{\rm I}}
\newcommand{\scrX}{\mathscr{X}}
\newcommand{\scrD}{\mathscr{D}}
\newcommand{\cP}{\mathcal{P}}
\newcommand{\cQ}{\mathcal{Q}}
\newcommand{\cS}{\mathcal{S}}

\def\bfL{\mathbf{L}}
\def\cF{\mathcal{F}}

\newcommand{\Orbit}{\mathcal{O}}

\newcommand{\HS}{\mathcal{H}}
\newcommand{\Co}{\mathrm{Co}}

\newcommand{\Rspan}[1]{\mathbb{R}\text{-}\mathrm{span}\{ #1\}}

\newcommand{\subgr}{\leq}

\newcommand{\dimG}{d}

\newcommand{\M}[3]{M^{#1}_{#2}(\R^{#3})} 
\newcommand{\B}[3]{B^{#1}_{#2}(\R^{#3})}	 
\newcommand{\hB}[3]{\dot{B}^{#1}_{#2}(\R^{#3})} 
\newcommand{\DS}[5]{\mathcal{D}(\mathscr{#1}^{#2}, L^{#3}, \ell^{#4}_{#5})} 
\newcommand{\E}[3]{E^{#1}_{#2}(\R^{2 #3 + 1})}	 
\newcommand{\EE}[3]{\tilde{E}^{#1}_{#2}(\R^{2 #3 + 1})}	 

\newcommand{\wt}{\mathbf{w}}	 
\newcommand{\vt}{\mathbf{v}} 
\newcommand{\ut}{\mathbf{u}} 
\newcommand{\pt}{\mathbf{p}} 
\newcommand{\qt}{\mathbf{q}} 
\newcommand{\st}{\mathbf{s}} 
\newcommand{\hvt}{\mathbf{v}^{\H}} 
\newcommand{\hut}{\mathbf{u}^{\H}} 
\newcommand{\dut}{\mathbf{u}^D} 
\newcommand{\hdut}{\mathbf{u}^{hD}} 

\begin{document}
\title[$\H$-Modulation Spaces]{Heisenberg-Modulation Spaces \\ at the Crossroads of \\ Coorbit Theory and Decomposition Space Theory}

\author[V. Fischer]{V\'{e}ronique Fischer}
\address{Department of Mathematical Sciences, University of Bath, North Rd, Bath BA2 7AY, UK}
\email{v.c.m.fischer@bath.ac.uk}

\author[D. Rottensteiner]{David Rottensteiner}
\address{Faculty of Mathematics, University of Vienna,
Oskar-Morgenstern-Platz 1, 1090 Vienna, Austria}
\email{david.rottensteiner@univie.ac.at}

\author[M. Ruzhansky]{Michael Ruzhansky}
\address{Department of Mathematics, Imperial College London, London, 180 Queen's Gate, 
SW7 2AZ, United Kingdom}
\email{m.ruzhansky@imperial.ac.uk}


\begin{abstract}

We show that generalised time-frequency
shifts on the Heisenberg group $\mathbf{H}_n \cong \mathbb{R}^{2n+1}$, realised as a unitary
irreducible representation of a nilpotent Lie group acting on $L^{2}(\mathbf{H}_n)$, give rise to
a novel type of function spaces on $\mathbb{R}^{2n+1}$.
The representation we employ is the generic unitary
irreducible representation of
the $3$-step nilpotent Dynin-Folland group.

In doing so, we answer the question
whether representations of nilpotent Lie groups ever yield coorbit spaces distinct
from the classical modulation spaces $M^{\mathbf{p}, \mathbf{q}}_{\mathbf{v}}(\mathbb{R}^{2n+1}), n \in \mathbb{N}$, and also introduce
a new member to the zoo of decomposition spaces.

As our analysis and proof of novelty make heavy use of coorbit theory and decomposition space
theory, in particular Voigtlaender's recent contributions to the latter,
we give a complete classification of the unitary
irreducible representations of the Dynin-Folland group and also
characterise the coorbit spaces related to the non-generic unitary irreducible representations.

\end{abstract}

\maketitle


\section{Introduction}

The modulation spaces $\M{\pt, \qt}{\vt}{n}$, introduced in Feichtinger~\cite{Fei83}, are the prototypical and by far most well-studied examples of function spaces induced by a square-integrable unitary irreducible representation of a nilpotent Lie group. The representation which is used is the Schr\"{o}dinger representation of the Heisenberg group $\H$, but it often comes in disguise as the so-called time-frequency shifts $f \mapsto e^{2 \pi i \omega t} \hspace{1pt} f(t + x)$ on $\R^n$. Other frequent function spaces like the homogeneous Besov spaces $\hB{\pt, \qt}{\st}{n}$ also permit a representation-theoretic description, and a common framework for such representation-theoretic descriptions, the so-called coorbit theory, was developed in Feichtinger and Gr\"{o}chenig~\cite{fegr89}.

This paper is motivated by the following questions. Are there any square-integrable unitary irreducible representations of connected, simply connected nilpotent Lie groups, realised as some kind of generalised time-frequency shifts on $\R^n$, which give rise to function spaces, specifically coorbit spaces, on $\R^n$ that are different from the modulation spaces $\M{\pt, \qt}{\vt}{n}$? (As one may suspect, they should also be different from the homogeneous Besov spaces and other wavelet coorbit spaces.) And if we suspect there are, then precisely how can we prove their distinctness?
Unfortunately, coorbit theory itself does not offer any standard tools to compare or distinguish the function spaces within its framework. However, the prominent examples $\M{\pt, \qt}{\vt}{n}$ and $\hB{\pt, \qt}{\st}{n}$ also happen to be decomposition spaces and the recent 188-page paper Voigtlaender~\cite{vo16-1} provides novel methods of comparing them.  

Voigtlaender's machinery shall be the backbone for the introduction of a new family of function spaces which permit a thorough analysis from a representation-theoretic as well as from a decomposition space-theoretic viewpoint. The generalised time-frequency shifts we use are given by the projective generic representation of a specific connected, simply connected $3$-step nilpotent Lie group first considered in Dynin~\cite{Dyn1}. Its Lie algebra is generated by the left-invariant vector fields on the Heisenberg group $\H$ and multiplication by the $2n+1$ coordinate functions. Not surprisingly, the group turns out to be a semi-direct product $\R^{2n+2} \rtimes \H$. In his paper Dynin makes use of its projective generic representation to develop a Weyl-type quantization on $\H$. As Dynin was motivated by the quantization, his account on the group and its generic representation was not very explicit. Folland mentioned the paper \cite{Dyn1} by Dynin 
in a miscellaneous remark in his monograph \cite[p.~90]{FollPhSp}, 
saying that the group might be called 
`the Heisenberg group of the Heisenberg group'.
Some years later, in \cite{FollMeta}, 
Folland provided a rigorous account on such $3$-step nilpotent Heisenberg constructions, which he the called `meta-Heisenberg groups' of the underlying $2$-step nilpotent groups. 
Paying tribute to both its first introduction by Dynin 
and its explicit description by Folland, 
we will call it the \emph{Dynin-Folland group} $\HG{n}{2}$.

Since the generic representation of $\HG{n}{2}$ seems different enough to induce a new class of function spaces on $\R^{2n+1} \cong \H$, we give a description of $\HG{n}{2}$ and all its unirreps, classified via Kirillov's orbit method. This enables us to study the coorbit spaces of mixed $L^{\pt, \qt}$-type under the generic representation of $\HG{n}{2}$, which we call Heisenberg-modulation spaces, and under all the other representations. We thereby provide a complete classification and characterisation of all the coorbit spaces related to $\HG{n}{2}$. Our analysis is based on a merger of the coorbit-theoretic and decomposition spaces-theoretic desciptions of the spaces, which is owed to the semi-direct product structure of $\HG{n}{2} = \R^{2n+2} \rtimes \H$.

Let us remark that we are not the first to consider coorbits under representations of nilpotent Lie groups other than the Heisenberg group. In their paper~\cite{BeltitaModSpNil}, Ingrid and Daniel Belti\c{t}\u{a} have used Kirillov's orbit method and the work of Niels Vigand Pedersen (cf.~\cite{PedGeomQuantUEA, PedWeylNilpot}) to define `modulation spaces for unitary irreducible representations of general nilpotent Lie groups.' Although some of their methods are inspired by coorbit theory, their approach does not make explicit use of it because their representations are not necessarily square-integrable. The resulting spaces, some of which are coorbits, are called modulation spaces and the classical unweighted spaces $\M{\pt, \qt}{}{n}$ are among them. However, the spaces are used as an abstract tool since the focus of \cite{BeltitaModSpNil} and a series of subsequent papers (cf.~\cite{BeltitaMag, BeltitaContMag, BeltitaAlgSymbols}) lies on the mapping properties of Weyl-Pedersen-quantized pseudodifferential operators rather than function space theory. The papers give in fact no proof that any of the spaces are different from the modulation spaces $\M{\pt, \qt}{}{n}$ and neither is there a concrete description of any particular example. We therefore do not draw any more attention to this connection than mentioning that the unweighted versions of our new spaces are among the abstract spaces defined by \cite[Def.~2.15]{BeltitaModSpNil}.

\medskip

The paper is organised as follows. In Section~\ref{Section_HG} we briefly recall some facts about the Heisenberg group $\H$ and the Schr\"{o}dinger representation, which will be crucial for our approach to the Dynin-Folland group $\HG{n}{2}$ and for studying of the coorbit spaces related to $\HG{n}{2}$.

In Section~\ref{Dynin_Alg_Group} we give an explicit and elementary construction of the Dynin-Folland Lie algebra and group as well as the group's generic unitary irreducible representations.

In Section~\ref{Section_Classification_Orbit_Method} we classify all the unitary irreducible representations of the Dynin-Folland group by Kirillov's orbit method. As a by-product we show that all of them are square-integrable modulo the respective projective kernels and we provide convenient descriptions of the corresponding projective representations.

In Section~\ref{Section_HM} we introduce the (unweighted) Heisenberg-modulation spaces $\E{\pt, \qt}{}{n}$. Our first definition is intended to be rather intuitive and in analogy to the definition of modulation spaces in Gr\"{o}chenig's monograph \cite{gr01}: the $\E{\pt, \qt}{}{n}$-norms are computed as the mixed $L^{\pt, \qt}$-norms over $\HG{n}{2}/Z(\HG{n}{2})$ of the matrix coefficients of the generalised time-frequency shifts given by the projective generic representation of $\HG{n}{2}$ which is parameterised by `Planck's constant' $\lambda = 1$ . Our concrete realisation of the representation is equivalent to intertwining Dynin's representation with the Euclidean Fourier transform. The representation thus acts by $\H$-frequency shifts on $\widehat{\R}^{2n+1}$.

In Section~\ref{DS_E} we introduce a class of decomposition spaces on $\R^{2n+1}$ whose frequency covering of $\widehat{\R}^{2n+1}$ is governed by a discrete lattice subgroup of $\H$. These spaces are the natural candidates for an equivalent decomposition space-theoretic description of the Heisenberg-modulation spaces $\E{\pt, \qt}{}{n}$.

The main results of this paper are proved in Section~\ref{CS_E}.  In Subsection~\ref{GenCoorbits} we introduce the coorbits of mixed $L^{\pt, \qt}$-type under the generic representations of the Dynin-Folland group, which are equipped with a reasonable class of weights. We observe that the unweighted coorbits coincide with the Heisenberg-modulation spaces $\E{\pt, \qt}{}{n}$ from Section~\ref{Section_HM}, which motivates a second, more general definition of Heisenberg-modulation spaces $\E{\pt, \qt}{\st}{n}$. In Subsection~\ref{CS_DF} we provide a complete classification of all the coorbits related to $\HG{n}{2}$, which implies that all but the generic coorbits are modulation spaces over $\R^n$ or $\R^1$. Finally, in Subsection~\ref{Novelty} we show the equality of the Heisenberg-modulation spaces and the decomposition spaces from Subsection~\ref{DS_E}. Combining the decomposition-space description with the novel machinery of Voigtlaender~\cite{vo16-1} we prove the distinctness of the Heisenberg-modulation spaces $\E{\pt, \qt}{\st}{n}$ from the well-known classes of modulation spaces and (homogeneous as well as inhomogeneous) Besov spaces on $\R^{2n+1}$

\medskip

\noindent\textbf{Convention.}
The authors make the choice to use several letters in the latin alphabet multiple times each; distinct notions may be assigned the same letter yet in a different (distinguishable) font each. This serves the purpose of coherence with the several important sources on which this paper draws.

\section{The Heisenberg Group} \label{Section_HG}

In this section we recall some facts about the Heisenberg group $\H$ which will be crucial for our approach to the Dynin-Folland group $\HG{n}{2}$ and the study of the coorbit spaces related to $\HG{n}{2}$. We construct the Heisenberg group $\H$ 
as the meta-Heisenberg of $\R^n$, that is, following Folland~\cite{FollMeta}.
We will  see how the Schr\"{o}dinger representations is the natural representation via this construction and we will give some explicit formulas for the left- and right- invariant vector fields which will be needed later on.

\subsection{A Realisation of $\H$} 
\label{Subsection_Construction_HG}

Let us define the operators $P_j$ and $Q_k$, $j, k = 1, \ldots, n$,
acting on the Schwartz space $\SFG{\R^n}$ by 
\begin{eqnarray}
P_j f(x) &:=& \frac{1}{2 \pi i} \pa{}{f}{x_j} (x), \label{MomOp}
\\
Q_k f(x) &:=& x_k f(x), \label{PosOp}
\end{eqnarray}
where $f\in \SFG{\R^n}$ and $x\in\R^n$.

One checks easily that these operators are formally self-adjoint and that for any $j, k = 1, \ldots, n$,
\begin{equation}
[P_j,P_k] = [Q_j,Q_k] = 0, \spa [P_j,Q_k] = \frac{\delta_{j,k}}{2 \pi i} \hspace{2pt} \I,
 \label{CCR}
\end{equation}
where $\I$ denotes the identity operator.
Here we have used the usual convention 
that the commutator of two operators $A,B$ acting on $\SFG{\R^n}$
is $[A,B]:=AB-BA$.
The relations \eqref{CCR} are called 
the \emph{Canonical Commutation Relations} (CCR) or Heisenberg Commutation Relations.

Let us denote by $\langle 2 \pi i \hspace{2pt} P_j, 2 \pi i \hspace{2pt} Q_k \rangle$ the Lie algebra (over $\R$) of skew-adjoint operators on $\SFG{\R^n}$ generated by the operators $2 \pi i \hspace{2pt} P_j$ and $2 \pi i \hspace{2pt} Q_k$ (with Lie bracket given by the commutator bracket).
The CCR  show that 
$$
 (2 \pi i)^{-1} \hspace{2pt} \langle 2 \pi i \hspace{2pt} P_j, 2 \pi i \hspace{2pt} Q_k \rangle = 
\R P_1\oplus\ldots\oplus \R P_n\oplus 
\R Q_1\oplus\ldots\oplus \R Q_n\oplus \R \I,
$$
This Lie algebra
has dimension $2n+1$ and is 2-step nilpotent.
Moreover $\langle 2 \pi i \hspace{2pt} P_j, 2 \pi i \hspace{2pt} Q_k \rangle$
is isomorphic to the Heisenberg Lie algebra $\h$,
whose definition we now recall.

\begin{dfn}
The \emph{Heisenberg Lie algebra} $\h$ is
the real Lie algebra with underlying vector space $\R^{2n+1}$ 
endowed with the Lie bracket defined by
\begin{equation}
j,k = 1, \ldots n,
\ \left.
\begin{array}{r}
[X_{p_j},X_{p_k}] = [X_{q_j},X_{q_k}] = [X_{p_j},X_t] = [X_{q_j},X_t] 
= 0,
\\
 \ [ X_{p_j} , X_{q_k}] 
= \delta_{jk} X_t,
\end{array}\right\}
\label{Lie_Bracket_as_CCR}
\end{equation}
where 
$(X_{p_1}, \ldots, X_{p_n}, X_{q_1}, \ldots, X_{q_n}, X_t)$
denotes the standard basis of $\R^{2n+1}$.
\end{dfn}

The canonical Lie algebra isomorphism 
between $\langle 2 \pi i \hspace{2pt} P_j, 2 \pi i \hspace{2pt} Q_k \rangle$ and $\h$
 is  
\begin{equation}
\label{def_drho}
d\rho:\h \longrightarrow  \langle 2 \pi i \hspace{2pt} P_j, 2 \pi i \hspace{2pt} Q_k \rangle
\end{equation}
defined  by
$$
 d\rho(X_{p_j}) = 2 \pi i \hspace{2pt}P_j,
\quad d\rho (X_{q_k}) = 2 \pi i \hspace{2pt} Q_k, \quad j,k=1,\ldots,n
\quad
\mbox{and}\  d\rho(X_t) = 2 \pi i \hspace{2pt} \I.
$$

The Lie algebra $\h$ is  nilpotent of step 2
and its centre is $\R X_t$.
In standard coordinates
$$
(p, q, t) := (p_1, \ldots, p_n, q_1, \ldots, q_n, t),
$$
and similarly for $(\pr{p}, \pr{q}, \pr{t})$,
its Lie bracket given by  \eqref{Lie_Bracket_as_CCR} becomes
	\begin{align}
		[(p, q, t),(\pr{p},\pr{q},\pr{t})]\, := \,(0, 0, p\pr{q} - q\pr{p}) \label{Lie_Bracket_Coordinates}
	\end{align}
if $p\pr{q}$ abbreviates the standard inner product of $p$ and $\pr{q}$ on $\R^n$.

The \emph{Heisenberg group} $\H$ is 
the connected, simply connected Lie group 
corresponding to the Heisenberg Lie algebra $\h$.

Hence $\H$ is a nilpotent Lie group of step 2 and its centre is $\exp (\R X_t)$.
The group law of $\H$ may be given by the Baker-Campbell-Hausdorff formula, which we now recall for a general Lie group $G$ and corresponding Lie algebras $\Lie{g}$
(see, e.g., \cite[p.11,12]{CorwinGreenleaf}).
It reads
	\begin{align}
		\exp_{G}(X) \odot_{G} \exp_{G}(Y) &= \exp_{G}(X + Y + \frac{1}{2} [X, Y]_{\Lie{g}} + \frac{1}{12} ( [X, [X, Y]_{\Lie{g}}]_{\Lie{g}} \nn \\ 
		&\qquad - [Y, [X, Y]_{\Lie{g}}]_{\Lie{g}}) - \frac{1}{24} [Y, [X, [X, Y]_{\Lie{g}}]_{\Lie{g}}]_{\Lie{g}} + \ldots). \label{BCH-Formula}
	\end{align}
This formula always holds at least on a neighbourhood of the identity of $G$
and in fact whenever the series on the right hand side converges.
If $G$ is a connected simply connected nilpotent Lie group,
the exponential mapping $\exp_G:\Lie{g}\rightarrow G$ 
is a bijection and \eqref{BCH-Formula} holds on $\Lie{g}$ since the series on the right hand side is finite.
For the case $\Lie{g} = \h$ it yields
	\begin{align}
		\exp_{\H}(X) \odot_{\H} \exp_{\H}(Y) 
		= \exp_{\H}\left(X + Y +\frac{1}{2} [X,Y]\right) 
		\label{Heisenberg_Group_Law_BCH}
	\end{align}
for all $X, Y$ in $\h$. 

We now realise the Heisenberg group $\H$ using exponential coordinates. This means that we identify an element of $\H$ with an element of $\R^{2n+1}$ via
	\begin{align*}
		(p, q, t) = \exp_{\H}\bigl( \Sum{j}{1}{n} (p_j X_{p_j} + q_j X_{q_j}) + t X_t \bigr).
	\end{align*}
Hence, using this identification, 
the centre of $\H$ is $\{(0,0,t) \mid t\in \R\}$ and
the group law given by \eqref{Heisenberg_Group_Law_BCH} 
becomes
	\begin{align}
		(p, q, t) \odot_{\H} (\pr{p},\pr{q},\pr{t}) = \bigl( p+\pr{p} ,q+\pr{q}, t+\pr{t}+\frac{1}{2}(p\pr{q}-q\pr{p}) \bigr). \label{Heisenberg_Group_Law_Coordinates}
	\end{align}

Since the $\H$-Haar measure coincides with the Lebesgue measure on $\R^{2n+1}$, we can make further use of the latter coordinates and write the $\H$-Haar measure as $dp \, dq \, dt$. Consequently, $\L{\mathbf{r}}{\H} \cong \L{\mathbf{r}}{\R^{2n+1}}$ for all $\mathbf{r} \in \R^+$.

Furthermore, the identification $\H\cong \R^{2n+1}$ allows us to define 
$\SFG{\H} \cong \SFG{\R^{2n+1}}$.

\subsection{Left-Invariant Vector Fields} 
\label{LVF_Sublap}

Let us recall the definitions of 
 the left and right regular representations of
 an arbitrary unimodular Lie group $G$
 on $L^2(G)$: 

\begin{dfn} \label{LRRegRep}
The representations $L$ and $R$ of $G$ on $L^2(G)$ 
defined by
 $$
 \left(L(g) f\right) (g_1)=f(g^{-1} g_1) \quad\mbox{and}\quad
 \left(R(g)f\right) (g_1)=f(g_1g),
\quad g,g_1\in G, \ f\in L^2(G),
 $$
 are called the \emph{left and right regular representations} of $G$ on $L^2(G)$,
 respectively.
\end{dfn}

Naturally, the left and right regular representations of $G$ on $L^2(G)$ are unitary
and their infinitesimal representations yield the isomorphisms between
the Lie algebra of $G$ and 
the Lie algebra of the smooth right- and left-invariant vector fields on $G$, respectively.
More precisely, 
the left-invariant vector field $d R(X)$
corresponding to a vector $X \in \Lie{g}$ at a point $g \in G$ is  given by
	\begin{align*}
		d R(X) f(g) = \dd{}{}{\tau}\bigg|_{\tau = 0} f(g \exp_{\H} (\tau X)),
	\end{align*} 
for any differentiable function $f$ on $G$, 
whereas the right-invariant vector field $d L(X)$ corresponding to $X$ is given by
	\begin{align*}
		d L (X) f(g) = \dd{}{}{\tau}\bigg|_{\tau = 0} f(\exp_{\H} (-\tau X) g).
	\end{align*}

Short computations  in the case of the Heisenberg group $\H$ 
 yield the following expressions
for the left and right-invariant vector fields corresponding to the basis vectors 
$X_{p_j}, X_{q_k}, X_t$
 for $j, k = 1, \ldots n$.
The left-invariant vector fields are given by
 	\begin{align*}
		dR(X_{p_j}) = \left( \pa{}{}{p_j} - \frac{1}{2} q_j \pa{}{}{t}  \right), 
		\spa dR (X_{q_k}) = \left( \pa{}{}{q_k} + \frac{1}{2} p_k \pa{}{}{t} \right), 
		\spa dR( X_{t}) = \pa{}{}{t},
	\end{align*}
the right-invariant vector fields by
	\begin{align*}
	-dL(X_{p_j}) = \left( \pa{}{}{p_j} + \frac{1}{2} q_j \pa{}{}{t}  \right), \spa
	-dL(X_{q_k}) = \left( \pa{}{}{q_k} - \frac{1}{2} p_k \pa{}{}{t} \right), \spa 
	-dL(X_{t}) = \pa{}{}{t}.
	\end{align*}

\subsection{The Schr\"{o}dinger Representation} 
\label{Unirreps_H1n}

Here we show that there is only one possible representation of  the Heisenberg group $\H$ 
with infinitesimal representation $d\rho$
defined by \eqref{def_drho}.
This `natural' representation $\rho$ 
will turn out to be the well known 
(canonical) Schr\"{o}dinger representation of $\H$.

We start with the following three observations.
Firstly, from the group law, we have
\begin{eqnarray*}
(p,q,t)
&=&
(0,q,0)(p,0,0)(0,0,t+\frac {pq}2)
\\&=&
\exp_{\H} (q_1 X_{q_1})
\ldots
\exp_{\H} (q_n X_{q_n})
\exp_{\H} (p_1 X_{p_1})
\ldots
\exp_{\H} (p_n X_{p_n})\\
&&\qquad
\exp_{\H} ((t+\frac {pq}2)X_t).
\end{eqnarray*}
Secondly, from the definition of an infinitesimal representation,
we know that
if $d\rho$ is the infinitesimal representation of $\rho$,
then we must have
$$
\rho (\exp_{\H}(\tau X )) =e^{\tau  d\rho ( X )}
$$
for every $X\in \h, \tau\in \R$, where the right hand side is understood 
as the strongly continuous 1-parameter group of unitary operators with generator $d\rho ( X)$ defined by Stone's theorem.
Therefore, if it can be constructed, the representation $\rho$ will be characterised by the 1-parameter groups with generators
$$
d\rho (X_{p_j})= 2 \pi i \hspace{2pt} P_j = 
\frac{\partial}{\partial x_j},
\quad
d\rho ( X_{q_k}) = 2 \pi i \hspace{2pt} Q_k = \times (2 \pi i \hspace{2pt} x_k)
\quad\mbox{and}\quad
d\rho ( X_t) = 2 \pi i \hspace{2pt} \I.
$$
Thirdly, it is well known that the operators $2 \pi i \hspace{2pt} P_j, 2 \pi i \hspace{2pt} Q_k$ and $2 \pi i \hspace{2pt} \I$
are defined on $\SFG{\R^n}$ 
but have essentially skew-adjoint extensions on $L^2(\R^n)$,
and generate the 1-parameter unitary groups of operators on $L^2(\R^n)$,
$$
\bigl\{ e^{d\rho (\tau X_{p_j})} \bigr\}_{\tau\in \R}, \quad
\bigl\{ e^{d\rho (\tau X_{q_k} )} \bigr\}_{\tau\in \R},\quad 
\bigl\{ e^{d\rho (\tau X_t)} \bigr\}_{\tau\in \R},
$$ 
given respectively by
\begin{eqnarray*}
e^{d\rho (\tau X_{p_j})}f(x) &=&f(x_1,\ldots,x_j+\tau,\ldots x_n),\\
e^{d\rho (\tau X_{q_k})} f(x) &=& e^{2 \pi i \tau x_k}f(x),\\
e^{d\rho (\tau X_t)}f(x) &=& e^{2 \pi i \tau } f(x),
\end{eqnarray*}
for $f\in L^2(\R^n)$, $x\in \R^n$.

From the three observations above, 
 the unique candidate $\rho$ for a representation of $\H$
having infinitesimal representation $d\rho$ must satisfy
\begin{align*}
\rho \left(\exp_{\H} (p_j X_{p_j})\right)f(x)
&=
e^{d\rho (p_j X_{p_j})}f(x) 
=
f(x_1,\ldots, x_j+p_j, \ldots x_n),\\
\rho \left(\exp_{\H} (q_k X_{q_k})\right)f(x)
&=
e^{d\rho (q_k X_{q_k})} f(x)
=
e^{2 \pi i q_kx_k}f(x),\\
\rho \left(\exp_{\H} (t X_{t})\right)f(x)
&= 
e^{d\rho (tX_t)}f(x) 
=
e^{2 \pi it} f(x),
\end{align*}
 for $f\in \SFG{\R^n}$ and $x\in \R^n$,
and we must have
\begin{align*}
\rho(p,q,t)f(x)
&=
e^{d\rho (q_1 X_{q_1})}
\ldots e^{d\rho (q_n X_{q_n})}
e^{d\rho (p_1 X_{p_1})}
\ldots e^{d\rho (p_n X_{p_n})}
e^{d\rho ((t+\frac {pq}2)X_t)} f(x)
\\&=
e^{2 \pi i qx} \left(e^{d\rho (p_1 X_{p_1})}
\ldots e^{d\rho (p_n X_{p_n})}
e^{d\rho ((t+\frac {pq}2)X_t)}
\right) f(x)
\\
&=
e^{2 \pi i qx} \left(e^{d\rho (0,0,t+\frac {pq}2)}\right) f(x+p)
\\
&=
e^{2 \pi i qx} e^{2 \pi i(t+\frac {pq}2)} f(x+p),
\end{align*}
that is,
\begin{equation}
\rho(p,q,t)f(x)
=
 e^{2 \pi i(t+qx +\frac {pq}2)} f(x+p).
\label{def_rho}
\end{equation}

Conversely, one checks easily that 
the expression $\rho$ defined by \eqref{def_rho}
gives a unitary representation of $\H$.
In fact we recognise the so-called \emph{Schr\"{o}dinger representation} of $\H$.

Let us recall that there is an intimate connection between the Schr\"{o}dinger representation $\rho$ and the time-frequency shifts used in time-frequency analysis. Following the convention of Gr\"{o}chenig's monograph~\cite{gr01}, a generic time-frequency shift on $\L{2}{\R^n}$ is the unitary operator defined as the composition $M_q T_p$ of a translation
	\begin{align*}
		(T_p f)(x) := f(x+p) = \bigl( \rho(p, 0, 0) f \bigr)(x)
	\end{align*}
and a modulation
	\begin{align*}
		(M_q f)(x) := e^{2 \pi i qx} \hspace{2pt} f(x) = \bigl( \rho(0, q, 0) f \bigr)(x).
	\end{align*}
The family $\{ M_q T_p \}_{(p, q, 0) \in \H}$ in fact coincides with the unitarily equivalent version of the Schr\"{o}dinger representation which is realised in the natural coordinates of the semi-direct product $\H = \R^{n+1} \rtimes \R^n$ and then quotiented by the central variable $t$. We will be making use of this alternative representation from Section~\ref{Section_HM} onwards.

\subsection{The Family of Schr\"{o}dinger Representations}
\label{subsec_family_Schrodinger}

In this section 
we describe the complete family of Schr\"{o}dinger representations 
$\rho_\lambda$, $\lambda\in \R\backslash \{0\}$, of $\H$.

In this paper 
we prefer to define a Lie algebra or a Lie group 
via a concrete description (the most common realisation or the most useful for a certain purpose) 
rather than as a class of isomorphic objects given by a representative.
Indeed, we have defined the Heisenberg Lie algebra $\h$ 
via the CCR on the standard basis of $\R^{2n+1}$
and we have considered a concrete realisation of the Heisenberg group $\H$.
However, it is interesting to define other isomorphisms than $d\rho$.
Indeed, let us consider the linear mapping 
$d\rho_{\lambda}:\h \rightarrow  \langle 2 \pi i \hspace{2pt} P_j, 2 \pi i \hspace{2pt} Q_k \rangle$
defined by
$$
d\rho_{\lambda}(X_{p_j}) = 2 \pi i \hspace{2pt} P_j,
\quad d\rho_{\lambda} (X_{q_k}) = \lambda 2 \pi i \hspace{2pt} \hspace{2pt} Q_k, \quad j,k=1,\ldots,n
\quad
\mbox{and}\  d\rho_{\lambda}(X_t) =  \lambda \hspace{2pt} 2 \pi i \hspace{2pt} \I,
$$
for a fixed $\lambda\in \R\backslash\{0\}$.

Proceeding as for $\rho$,
the following property is easy to check:

\begin{lem}
\label{lem_rholambda}
For each $\lambda\in \R\backslash\{0\}$, 
the mapping $d\rho_\lambda $ is a Lie algebra isomorphism from $\h$ 
onto $\langle 2 \pi i \hspace{2pt} P_j, 2 \pi i \hspace{2pt} Q_k \rangle$.
It is the infinitesimal representation of the unitary representation $\rho_\lambda$ of $\H$ on $L^2(\R^n)$ given by
$$
\rho_\lambda (p,q,t)f(x)
=
 e^{2 \pi i \lambda(t+qx +\frac{1}{2}pq)} f(x+p),
 $$
for $f\in L^2(\R^n)$, $x\in \R^n$.

Naturally $\rho=\rho_1$.
\end{lem}

The representations $\rho_\lambda$
given by Lemma \ref{lem_rholambda} 
are called the Schr\"{o}dinger
representations with parameter $\l \in \R\backslash\{0\}$.
A celebrated theorem of Stone and von Neumann says that,
up to unitary equivalence,
these are all the irreducible unitary representations of $\H$ that are nontrivial on the centre:		
		\begin{thm}[Stone-von Neumann]\label{StvNThm}
For any $\lambda\in \R\backslash\{0\}$,
the representation 
 $\rho_\lambda$ of $\H$ is unitary and irreducible.
  If $\lambda,\lambda'\in \R\backslash\{0\}$ 
		with $\lambda\not=\lambda'$,
		then
the representations $\rho_\lambda$ and $\rho_{\lambda'}$
are  inequivalent.
Moreover, if $\pi$ is an irreducible and unitary representation of $\H$ 
such that $\pi(0,0,t)=e^{i\lambda t}$ for some $\lambda\not=0$, 
then $\pi$ is unitarily equivalent to $\rho_\lambda$.
		\end{thm}

For a proof see, e.g., \cite[Ch.~1~\S~5]{FollPhSp}.

It is not difficult to construct other realisations of the Schr\"{o}dinger representations defined above. For example, let us define the mapping $\tilde \rho_\lambda$ of $\H$ by
$$
\tilde \rho_\lambda(p,q,t):= 
\rho (\sqrt {|\lambda|}p, \frac{\lambda}{\sqrt{|\lambda|}}q,\lambda t).
$$
We check readily that it is a unitary irreducible representation of $\H$ on $L^2(\R^n)$. We can show easily that it is unitarily equivalent to $\rho_\l$ by direct computations or, alternatively, using the Stone-von Neumann theorem and checking that it coincides on the centre of $\H$ with the character $e^{2 \pi i\lambda \cdot}$.

The Stone-von Neumann Theorem  gives an almost complete classification of the $\H$-unirreps. In fact, we see that the only other unirreps 
which can appear are trivial at the centre. 
Passing the centre through the quotient,
those representations are now unirreps of the Abelian group $\R^{2n}$, 
hence characters of $\R^{2n}$. 
We thus have:

	\begin{thm}  [Classification of $\H$-Unirreps] \label{CharIrrUnitRep}
Every irreducible unitary representation $\rho$ of $\H$ on a Hilbert space $\HS$ is unitarily equivalent to one and only one of the following representations:
	\begin{itemize}
		\item[(a)] $\sigma_{(a,b)} \mid (p, q, t)\mapsto e^{2 \pi i(aq+bp)}$, $a,b\in\R^n$, acting on $\CF$,
		\item[(b)] $\rho_\lambda$, $\lambda \in \R \setminus \{ 0 \}$, acting on $\L{2}{\R^n}$.
	\end{itemize}
	\end{thm}


\section{The Dynin-Folland Lie Algebra and Lie Group} \label{Dynin_Alg_Group}

This section is dedicated to giving an explicit and elementary approach to the Dynin-Folland Lie algebra and group 
as well as the generic representations associated with the construction.

\subsection{The Lie Algebra $\HA{n}{2}$} \label{Section_h2n}

In this subsection we study the real Lie algebra \\
$\langle 2 \pi i \hspace{2pt} \scrD_{p_j}, 2 \pi i \hspace{2pt} \scrD_{q_k}, 2 \pi i \hspace{2pt} \scrD_t, 2 \pi i \hspace{2pt} \scrX_{p_l}, 2 \pi i \hspace{2pt} \scrX_{q_m}, 2 \pi i \hspace{2pt} \scrX_{t} \rangle$ generated by the left-invariant vector fields
\begin{equation}
\left\{ \begin{array}{rcccl}
	\scrD_{p_j} & := & (2 \pi i)^{-1} \hspace{2pt} dR(X_{p_j}) & = & (2 \pi i)^{-1} \hspace{2pt} \left( \pa{}{}{p_j} - \frac{1}{2} q_j \pa{}{}{t}  \right), \\
	\scrD_{q_k} & := & (2 \pi i)^{-1} \hspace{2pt} dR (X_{q_k}) & = &(2 \pi i)^{-1} \hspace{2pt} \left( \pa{}{}{q_k} + \frac{1}{2} p_k \pa{}{}{t} \right), \\
	\scrD_{t} & := & (2 \pi i)^{-1} \hspace{2pt} dR( X_{t}) & = & (2 \pi i)^{-1} \hspace{2pt} \pa{}{}{t},
\end{array}\right.
	\label{Left_VF_HG}
\end{equation}
%
and the multiplications by coordinate functions
\begin{equation} \label{HG_Coordinate_Mult}
\left\{ \begin{array}{rcl}
\scrX_{p_l} f\, (p,q,t)&=& p_l \hspace{1pt} f(p,q,t),\\
\scrX_{q_m} f\, (p,q,t)&=& q_m \hspace{1pt} f(p,q,t),\\
\scrX_{t} f\, (p,q,t)&=& t \hspace{1pt} f(p,q,t),\\
\end{array}\right.
\end{equation}
where $j, k, l, m = 1, \ldots n$ and $f\in \SFG{\H}$.
To this end, we compute all possible commutators between these operators, up to skew-symmetry. 
The symbol $\I$ will denote the identity operator on $\L{2}{\H}$.

Since the scalar multiplication operators commute, we have
$$
[(2 \pi i) \hspace{2pt} \mathscr{X}_{p_j}, (2 \pi i) \hspace{2pt} \mathscr{X}_{q_k}] = [(2 \pi i) \hspace{2pt} \mathscr{X}_{p_j}, (2 \pi i) \hspace{2pt} \mathscr{X}_t] = [(2 \pi i) \hspace{2pt} \mathscr{X}_{q_k}, (2 \pi i) \hspace{2pt} \mathscr{X}_t] = 0.
$$

The commutator brackets between the left invariant vector fields 
$\scrD_{p_j}, \scrD_{q_k}, \scrD_t$, for $j, k \in \{ 1, \ldots, n \}$, can be computed directly using \eqref{Left_VF_HG}:
\begin{equation*}
	\begin{array}{lcl}
(2 \pi i)^2 \, [\scrD_{p_j}, \scrD_{q_k}] & = & [\partial_{p_j} - \frac{1}{2} q_j \partial_t, \partial_{q_k} + \frac{1}{2} p_k \partial_t]  =  [\partial_{p_j}, \frac{1}{2} p_k \partial_t] + [- \frac{1}{2} q_j \partial_t, \partial_{q_k}] \\
& = & \frac{1}{2}  \delta_{j,k} \partial_t + \frac{1}{2}  \delta_{k,j} \partial_t  =  \delta_{j,k} \, \partial_t = 2 \pi i \delta_{j,k} \, \scrD_t, \\
(2 \pi i)^2 \, [\scrD_{p_j}, \scrD_t] & = & [\partial_{p_j} - \frac{1}{2} q_j \partial_t, \partial_t]  =  0, \\
(2 \pi i)^2 \, [\scrD_{q_j}, \scrD_t] & = & [\partial_{q_j} + \frac{1}{2} p_j \partial_t, \partial_t]  =  0.
	\end{array}
\end{equation*}

Naturally, we obtain that the operators $\scrD_{p_j}, \scrD_{q_k}, \scrD_t$ satisfy the CCR since the space of left-invariant vector fields on $\H$ forms a Lie algebra of operators isomorphic to $\h$ (see Section~\ref{LVF_Sublap}).
Let us compute the commutator brackets between the left-invariant vector fields and the coordinate operators, first the commutators with 
$\scrD_{p_j}$:

\begin{equation*}
	\begin{array}{lclclclcl}
(2 \pi i) \hspace{2pt} [\scrD_{p_j}, \mathscr{X}_{p_k}] & = & [\partial_{p_j} - \frac{1}{2} q_j \partial_t, p_k] & = & [\partial_{p_j}, p_k] &  = & \delta_{j,k} \, I, \\ 
(2 \pi i) \, [\scrD_{p_j}, \mathscr{X}_{q_k}] & = & [\partial_{p_j} - \frac{1}{2} q_j \partial_t, q_k] & = & 0, &&\\
(2 \pi i) \,[\scrD_{p_j}, \mathscr{X}_t] & = & [\partial_{p_j} - \frac{1}{2} q_j \partial_t, t] & = & [- \frac{1}{2} q_j \partial_t, t] & = & - \frac{1}{2} q_j & = & - \frac{1}{2} \mathscr{X}_{q_j},  
	\end{array}
\end{equation*}
then with $\scrD_{q_j}$:
\begin{equation*}
	\begin{array}{lclclclcl}
(2 \pi i) \,[\scrD_{q_j}, \mathscr{X}_{p_k}] & = & [\partial_{q_j} + \frac{1}{2} p_j \partial_t,  p_k] & = & 0, \\
(2 \pi i) \, [\scrD_{q_j}, \mathscr{X}_{q_k}] & = & [\partial_{q_j} + \frac{1}{2} p_j \partial_t,  q_k] & = & [\partial_{q_j}, q_k] & = & \delta_{j,k} \, \I,  \\
(2 \pi i) \, [\scrD_{q_j}, \mathscr{X}_t] & = & [\partial_{q_j} + \frac{1}{2} p_j \partial_t, t]  & = & [\frac{1}{2} p_j \partial_t, t] = \frac{1}{2} p_j & = & \frac{1}{2} \mathscr{X}_{p_j}, \\
	\end{array}
\end{equation*}
and eventually with $\scrD_t$:
\begin{equation*}
	\begin{array}{lclclclcl}
(2 \pi i) \, [\scrD_t, \mathscr{X}_{p_k}] & = & [\partial_t, p_k] & = & 0, \\
(2 \pi i) \, [\scrD_t, \mathscr{X}_{q_k}] & = & [\partial_t, q_k] & = & 0, \\
(2 \pi i) \, [\scrD_t, \mathscr{X}_t] & = & [\partial_t, t] & = & \I.
	\end{array}
\end{equation*}


We conclude that the linear space (over $\R$) generated by the first order commutator brackets between the operators $\scrD_{p_j}, \scrD_{q_j}, \scrD_t$ and $\scrX_{p_j}$, $\scrX_{q_k}$, $\scrX_{t}$ is $(2 \pi i)^{-1}$ times
	\begin{align*}
\R \scrD_t \oplus \R  \I \oplus \R \scrX_{q_1}\oplus \ldots \oplus \R \scrX_{q_n}\oplus \R \scrX_{p_1}\oplus \ldots \oplus \R \scrX_{p_n}.		
	\end{align*}

The whole lot of commutators tells us that very few second order commutators remain. 
More precisely, the Lie brackets of $\scrD_t$, $\scrX_{p_j}$ or $\scrX_{q_k}$ with any $\scrD_{p_{j'}}, \scrD_{q_{j'}}, \scrD_t$ and 
$\scrX_{p_{j'}}$, $\scrX_{q_{k'}}$, $\scrX_{t}$
 can only vanish or be equal to $\I$, and  the operator $\I$ clearly commutes with all operators, hence does not create any new structure.
Therefore, the second order commutator brackets are all proportional to $\I$ and all third order commutators must  be zero.
We have obtained:

\begin{lem}
The real Lie algebra $\langle 2 \pi i \hspace{2pt} \scrD_{p_j}, 2 \pi i \hspace{2pt} \scrD_{q_k}, 2 \pi i \hspace{2pt} \scrD_t, 2 \pi i \hspace{2pt} \scrX_{p_l}, 2 \pi i \hspace{2pt} \scrX_{q_m}, 2 \pi i \hspace{2pt} \scrX_{t} \rangle$ generated by the operators \eqref{Left_VF_HG} and \eqref{HG_Coordinate_Mult} equals 
	\begin{align*}
		\R 2 \pi i \hspace{2pt} \scrD_{p_1}\oplus& \ldots \oplus \R 2 \pi i \hspace{2pt} \scrD_{p_n} \oplus
\R 2 \pi i \hspace{2pt} \scrD_{q_1} \oplus \ldots \oplus\R 2 \pi i \hspace{2pt} \scrD_{q_n} \oplus \R 2 \pi i \hspace{2pt} \scrD_t 
\oplus \\ \R 2 \pi i \hspace{2pt} \scrX_{p_1} \oplus& \ldots \oplus \R 2 \pi i \hspace{2pt} \scrX_{p_n} \oplus \R 2 \pi i \hspace{2pt} \scrX_{q_1}\oplus \ldots \oplus \R 2 \pi i \hspace{2pt} \scrX_{q_n} \oplus \R 2 \pi i \hspace{2pt} \scrX_{t} \oplus \R 2 \pi i \hspace{2pt} \I.
	\end{align*}
In particular, the identity operator $\I$ is the only newly generated element and spans the centre of the Lie algebra. 
Furthermore, this Lie algebra is $3$-step nilpotent and of topological dimension $2 (2n + 1) + 1$. 
\end{lem}

We now define the `abstract' Lie algebra that will naturally be isomorphic to \\ $\langle 2 \pi i \hspace{2pt} \scrD_{p_j}, 2 \pi i \hspace{2pt} \scrD_{q_k}, 2 \pi i \hspace{2pt} \scrD_t, 2 \pi i \hspace{2pt} \scrX_{p_l}, 2 \pi i \hspace{2pt} \scrX_{q_m}, 2 \pi i \hspace{2pt} \scrX_{t} \rangle$.
First we index the  standard basis of $\R^{2 (2n + 1) + 1}$ by
	\begin{align*}
(X_{u_1}, \ldots, X_{u_n}, X_{v_1}, \ldots, X_{v_n}, X_w, X_{x_1}, \ldots, X_{x_n}, X_{y_1}, \ldots, X_{y_n}, X_z, X_s).
	\end{align*}
Then we consider the linear isomorphism 
\begin{equation}
\label{def_dpi}
d\pi: \R^{2 (2n + 1) + 1} 
\longrightarrow 
\langle 2 \pi i \hspace{2pt} \scrD_{p_j}, 2 \pi i \hspace{2pt} \scrD_{q_k}, 2 \pi i \hspace{2pt} \scrD_t, 2 \pi i \hspace{2pt} \scrX_{p_l}, 2 \pi i \hspace{2pt} \scrX_{q_m}, 2 \pi i \hspace{2pt} \scrX_{t} \rangle
\end{equation}
defined by
$$
\begin{array}{l lll}
d\pi(X_{u_j}) = 2 \pi i \, \scrD_{p_j},&
d\pi(X_{v_j})= 2 \pi i \, \scrD_{q_j},&
d\pi(X_w)= 2 \pi i \, \scrD_t,\\
d\pi(X_{x_j})= 2 \pi i \, \scrX_{p_j},&
d\pi(X_{y_j})= 2 \pi i \, \scrX_{q_j},& \\
d\pi(X_z)= 2 \pi i \, \scrX_t, 
& d\pi(X_s)= 2 \pi i \, \I.\\
\end{array}
$$
	\begin{dfn}
We denote by $\HA{n}{2}$ the real Lie algebra with underlying linear space 
$\R^{2 (2n + 1) + 1}$ 
and Lie bracket $[\cdot,\cdot]_{\HA{n}{2}}$
defined such that $d\pi$ is a Lie algebra morphism.
	\end{dfn}
This means that 
the vectors in the standard basis of $\R^{2 (2n + 1) + 1}$ satisfy the 
 following commutator relations
\begin{equation}
\left\{ \begin{array}{rcr}
[X_{u_j}, X_{v_k}]_{\HA{n}{2}} &=& \delta_{j,k} \, X_w, \\
\, [X_{u_j}, X_{x_k}]_{\HA{n}{2}} &=& \delta_{j,k} \, X_s, \\
\, [X_{u_j}, X_z]_{\HA{n}{2}} &=& - \frac{1}{2} X_{y_j},  \\
\, [X_{v_j}, X_{y_k}]_{\HA{n}{2}} &=& \delta_{j,k} \, X_s,  \\
\, [X_{v_j}, X_z]_{\HA{n}{2}} &=& \frac{1}{2} X_{x_j}, \\
\, [X_w, X_z]_{\HA{n}{2}} &=& X_s. \\
\end{array}\right.
\label{CRG}
\end{equation} 
In \eqref{CRG}, we have only listed the non-vanishing Lie brackets of $\HA{n}{2}$,
up to skew-symmetry.

Our choice of notation $\HA{n}{2}$ for the Lie algebra 
reflects the fact that we just have applied a further type of Heisenberg construction to $\h$.
We will refer to $\HA{n}{2}$ as the Dynin-Folland Lie algebra in recognition of Dynin's and Folland's works \cite{Dyn1, Dyn2} and \cite{FollMeta}, respectively.

To conclude this subsection, we summarize a few properties of $\HA{n}{2}$. Before we list them, however, let us recall the notion of strong Malcev basis of a nilpotent Lie algebra. The existence of such a type of basis for every nilpotent Lie algebra is granted by the following theorem. See, e.g.,~\cite[Thm.1.1.13]{CorwinGreenleaf} for a proof.

	\begin{thm} \label{StrongMB}
Let $\Lie{g}$ be a nilpotent Lie algebra of dimension $\dimG$ and let
$\Lie{g}_1 \subgr \ldots \subgr  \Lie{g}_l \subgr \Lie{g}$ be ideals with
$\dim(\Lie{g}_j) = m_j$. Then there exists a basis $\{ X_1, \ldots,
X_\dimG \}$ such that 
	\begin{itemize}
		\item[(i)] for each $1 \leq m \leq \dimG$, $\Lie{h}_m := \Rspan{X_1, \ldots, X_m}$ is an ideal of $\Lie{g}$;
		\item[(ii)] for $1 \leq j \leq l$, $\Lie{h}_{m_j} = \Lie{g}_j$.
	\end{itemize}
A basis satisfying (i) and (ii)  is called a strong Malcev basis
of $\Lie{g}$ passing through the ideals $\Lie{g}_1, \ldots, \Lie{g}_l
$.
	\end{thm}

\begin{prop} \label{prop}
	\begin{itemize}
		\item[(i)] The Lie algebra $\HA{n}{2}$ is nilpotent of step 3, with centre $\Lie{z}(\HA{n}{2}) = \R X_s$.
		\item[(ii)] The mapping $d\pi$ is a morphism from the Heisenberg Lie algebra $\HA{n}{2}$ onto \\
		$\langle 2 \pi i \hspace{2pt} \scrD_{p_j}, 2 \pi i \hspace{2pt} \scrD_{q_k}, 2 \pi i \hspace{2pt} \scrD_t, 2 \pi i \hspace{2pt} \scrX_{p_l}, 2 \pi i \hspace{2pt} \scrX_{q_m}, 2 \pi i \hspace{2pt} \scrX_{t} \rangle$.
		\item[(iii)] The subalgebra $\langle 2 \pi i \hspace{2pt} \scrD_{p_j}, 2 \pi i \hspace{2pt} \scrD_{q_k}, 2 \pi i \hspace{2pt} \scrD_t \rangle$
		is isomorphic to the Heisenberg Lie algebra $\h$, and so is the subalgebra 
		$\R X_{u_1}\oplus\ldots\oplus \R X_{v_n}\oplus \R X_w$.
		Furthermore, the restriction of $d\pi$ to the subalgebra 
		$\R X_{u_1}\oplus\ldots\oplus \R X_{v_n}\oplus \R X_w$ coincides with the 
		infinitesimal right regular representation of $\H$ on $L^2(\R^n)$.
		\item[(iv)] The subalgebra $\langle 2 \pi i \hspace{2pt} \scrX_{p_l}, 2 \pi i \hspace{2pt} \scrX_{q_m}, 2 \pi i \hspace{2pt} \scrX_{t}, 2 \pi i \hspace{2pt} I \rangle$ is Abelian 
		and so is the subalgebra $\Lie{a} := \R X_{x_1} \oplus \ldots \oplus \R X_{y_n} \oplus \R X_z \oplus \R X_s$.
		\item[(v)] The basis $\{ X_s, X_{y_n}, \ldots, X_{x_1}, X_z, X_w, X_{v_n}, \ldots, X_{u_1} \}$ is a strong Malcev basis of $\HA{n}{2}$ which passes through the ideals $\Lie{z}(\HA{n}{2})$ and $\Lie{a}$.
	\end{itemize}
\end{prop}

\subsection{The Lie Group $\HG{n}{2}$} \label{Section_H2n}

Here we describe the connected simply connected $3$-step nilpotent Lie group 
that we obtain by exponentiating the Dynin-Folland Lie algebra 
$\HA{n}{2}$.
We denote this group by $\HG{n}{2}$.

As in the case of the Heisenberg group (cf.~Subsection~\ref{Subsection_Construction_HG})
 we can again make use of the Baker-Campbell-Hausdorff formula 
 recalled in \eqref{BCH-Formula}.
 Since the Dynin-Folland Lie algebra is of step $3$,
we obtain the group law
$$
\exp_{\HG{n}{2}}(X) \odot_{\HG{n}{2}} \exp_{\HG{n}{2}}(X') 
=
\exp_{\HG{n}{2}}(Z),
$$
with
\begin{equation}
Z:=
X + X' + \frac{1}{2} [X, X']_{\HA{n}{2}} 
 + \frac{1}{12} [(X-X'), [X, X']_{\HA{n}{2}}]_{\HA{n}{2}}  . \label{H2n_GL_BCH_Version}
\end{equation}
Let us compute $Z$ more explicitly. 
We write 
$$
X= \Sum{j}{1}{n} (u_j X_{u_j} + v_j X_{v_j}) + w X_w + \Sum{j}{1}{n} (x_j X_{x_j} + y_j X_{y_j}) + z X_z + s X_s, 
$$
and similarly for $X'$.
As in the Heisenberg case,
we abbreviate for instance sums like $\Sum{j}{1}{n} u_j X_{u_j}$ by the dot-product-like notation $u X_u$. Consequently we have
$$
		X = u X_u + v X_v + w X_w + x X_x + y X_y + z X_z + s X_s .
		$$
We compute readily (cf.~\cite[Lem.~3.4]{Ro14})

\begin{lem}
With the notation above, 
the expression of $Z$ given in \eqref{H2n_GL_BCH_Version}
becomes
	\begin{align*}
&Z=
 (u + \pr{u}) X_u + (v + \pr{v}) X_v + 
 \Bigl( w + \pr{w} + \frac{u \pr{v} - v \pr{u}}2 \Bigr) X_w 
\nonumber\\
& \qquad + \Bigl( x + \pr{x} + \frac14 (\pr{z} v  - z \pr{v})\Bigr) X_x + \Bigl( y + \pr{y} - \frac14 (\pr{z} u  - z \pr{u}) \Bigr) X_y + (z + \pr{z}) X_z 
\nonumber \\
& \qquad + \Bigl( s + \pr{s} + \frac{u \pr{x} - x \pr{u}}2 + \frac{v \pr{y} - y \pr{v}}2 + \frac{w \pr{z} - z \pr{w}}2 
 - \frac{z - \pr{z}}8 (u \pr{v} - v \pr{u}) \Bigr) X_s. 
\end{align*}
\end{lem}


As in the case of the Heisenberg group, 
we identify an element of the group with an element 
of the underlying vector space $\R^{2(2n+1)+1}$ of the Lie algebra:
$$
(u, v, w, x, y, z, s) = \exp_{\HG{n}{2}} \Bigl( u X_u + v X_v + w X_w + x X_x + y X_y + z X_z + s X_s \Bigr).
$$

We conclude:

\begin{prop}
\label{prop_HG2n}
With the convention explained above, 
the centre of the  $\HG{n}{2}$ is 
$\exp_{\HG{n}{2}} (\R X_s)=\{(0,0,0,0,0,0,s) \mid s \in \R\}$,
and the group law becomes
 	\begin{align}
		(u, v&, w, x, y, z, s) \odot_{\HG{n}{2}} (\pr{u}, \pr{v}, \pr{w}, \pr{x}, \pr{y}, \pr{z}, \pr{s}) \nn \\
		 =& \Bigl( 
		 u + \pr{u} \ , \  v + \pr{v} \  , \
 w + \pr{w} + \frac{u \pr{v} - v \pr{u}}2 \ ,
\nn \\
		&\quad
		x + \pr{x} + \frac 14 (\pr{z} v  - z \pr{v}) \ , \
		y + \pr{y} - \frac 14 (\pr{z} u  - z \pr{u}) \ , \
		z + \pr{z}\ , \nn \\
		&\quad
s + \pr{s} + \frac{u \pr{x} - x \pr{u}}2 + \frac{v \pr{y} - y \pr{v}}2 + \frac{w \pr{z} - z \pr{w}}2 
 - \frac{z - \pr{z}}8 (u \pr{v} - v \pr{u})\Bigr). \label{DF_GrLw}
	\end{align}
Furthermore, the subgroup $\{(u,v,w,0,0,0,0) \mid u,v\in \R^n,\ w\in \R\}$ 
is isomorphic to the Heisenberg group $\H$.
\end{prop}

\subsection{An Extended Notation - Ambiguities and Usefulness} \label{Section_Notation}

In this Section, 
we introduce new  notation to be able to perform computations in a concise manner. 
Unfortunately,
this will mean on the one hand identifying many different objects 
and on the other hand having several ways for describing one and the same operation.
Yet, the nature of our situation requires it.

Having identified the groups $\H$ and $\HG{n}{2}$ 
with the underlying vector space
(via exponential coordinates), 
many computations involve the variables $p$, $q$, $t$, $u$, $v$, $w$, $x$, $y$, $z$, $s$, 
which may refer to elements or the components of elements of the Lie algebras $\R^n$, $\h$, $\HA{n}{2}$ as well as elements or components of elements of the Lie groups $\R^n$, $\H$, $\HG{n}{2}$. 
Certain specific calculations moreover involve sub-indices $j, k, l, \ldots = 1, \ldots, n$ of the latter, that is, the scalar variables $p_j, q_k, t, u_l, \ldots$. 
Yet other formulas become not only less cumbersome but more lucid 
if we also introduce capital letters to denote members of $\H \cong \h \cong \R^{2n+1}$ and calligraphic capital letters for either $\H$-valued or scalar-valued components of the $2 \, (2n + 1) + 1$-dimensional elements of $\HA{n}{2} \cong \HG{n}{2}$.

Let the standard variables that define the elements of the Heisenberg group $\H \cong \h \cong \R^{2n+1}$ once and for all be fixed to be
	\begin{align}
		X := (p, q, t) := (p_1, \ldots, p_n, q_1, \ldots, q_n, t),
		\label{not_X}
	\end{align}
and let the standard variables that define the elements of the Dynin-Folland group $\HG{n}{2} \cong \HA{n}{2} \cong \R^{2 \, (2n+1) + 1}$ be denoted by
	\begin{align}
		(\mathcal{P}, \mathcal{Q}, \mathcal{S}) &:=((u, v, w), (x, y, z), s) 
		\label{not_cP_cQ_cS}\\
		&:= ((u_1, \ldots, u_n, v_1, \ldots, v_n, w), (x_1, \ldots, x_n, y_1, \ldots, y_n, z), s). \nn
	\end{align}
This purely notational identification of elements belonging to Lie groups, Lie algebras and Euclidean vector spaces will prove very useful in many instances. The $\H$ and $\HG{n}{2}$-group laws, for example, can be expressed in a very convenient way.
Let expressions like $\pr{p}q$ or $u \pr{v}$ again denote the standard $\R^n$-inner products of the vectors $\pr{p}, q$ and $u, \pr{v}$, respectively, whereas $\R^{2n+1}$-inner products will be denoted by
	\begin{align*}
		\bracket{\, . \,}{\, . \,} := \subbracket{\R^{2n+1}}{\, . \,}{\, . \,}.
	\end{align*}
Moreover, let us introduce the `big dot-product'
	\begin{align*}
		X \HP \pr{X} := (p, q ,t) \HP (\pr{p}, \pr{q}, \pr{t})
	\end{align*}
for elements in $\h \cong \H \cong \R^{2n+1}$ as an abbreviation of the $\H$-product \eqref{Heisenberg_Group_Law_Coordinates}, and let us agree that for all such vectors we can employ the $\h$-Lie bracket notation
	\begin{align*}
		[X, \pr{X}] := [X, \pr{X}]_{\h} := (0, 0, p \pr{q} - q \pr{p}).
	\end{align*}
We can then rewrite the $\H$-group law as
\begin{eqnarray}
		X \odot_{\H} \pr{X} &=& (p, q ,t) \odot_{\H} (\pr{p}, \pr{q}, \pr{t}) = \bigl( p + \pr{p} ,q + \pr{q}, t + \pr{t} + \frac{1}{2} (p \pr{q}-q \pr{p}) \bigr) \nn \\
		&=& X \HP \pr{X}= X + \pr{X} + \frac{1}{2} \, [X, \pr{X}].
		\label{eq_Hgrouplaw}
\end{eqnarray}

Let us turn our attention to the group law of $\HG{n}{2}$.
The beginning of \eqref{DF_GrLw} can be rewritten as
	\begin{align*}
		\left( u + \pr{u}, v + \pr{v}, w + \pr{w} +\frac{1}{2} \, (u \pr{v} - v \pr{u}) \right) = \mathcal{P} \HP \pr{\mathcal{P}},
	\end{align*}
if $\cP=(u,v,w)$ and similarly for $\pr{\cP}$.
For the rest of the formula we need to introduce the operation
\begin{equation}
\label{not_coad}
{\coad}_{\H}(X)(\pr{X}) = (\pr{t} q, -\pr{t} p, 0),
\end{equation}
if $X=(p,q,t)$ as in \eqref{not_X} and similarly for $X'$.
As the notation suggests,
the operation ${\coad}_{\H}$ is the coadjoint representation, 
where $\h$ and its dual have been identified with $\R^{2n+1}$.

Using the notation explained above, direct computations yield the following two lemmas and one proposition (cf.~\cite{Ro14}, Lemmas~3.7 and 3.8 as well as Subsection~3.6, respectively).

\begin{lem} \label{Lem_Nice_Formulas_Product_LieBr}
With the convention explained above,
the product of two elements $(\cP,\cQ,\cS)$ and $(\cP',\cQ',\cS')$ 
in $\HG{n}{2} \cong \HA{n}{2}$ is
	\begin{align}
		&(\mathcal{P}, \mathcal{Q}, \mathcal{S}) \odot_{\HG{n}{2}} (\pr{\mathcal{P}}, \pr{\mathcal{Q}}, \pr{\mathcal{S}}) \nn\\
		&\quad = \bigl( \mathcal{P} \HP \pr{\mathcal{P}} \ , \
		\mathcal{Q} + \pr{\mathcal{Q}} + \frac14 ({\coad}_{\H}(\mathcal{P})(\pr{\mathcal{Q}}) - {\coad}_{\H}(\pr{\mathcal{P}})(\mathcal{Q}))\ , \nn \\
		&\qquad \qquad  \mathcal{S} + \pr{\mathcal{S}} + \frac12 \bigl(\bracket{\mathcal{P}}{\pr{\mathcal{Q}}} - \bracket{\mathcal{Q}}{\pr{\mathcal{P}}} \bigr) - \frac18 \bracket{\mathcal{Q} - \pr{\mathcal{Q}}}{[\mathcal{P}, \pr{\mathcal{P}}]} \bigr),\label{eq_group_law_cal}
		\end{align}
while their Lie bracket is given by
	\begin{align}
		[(\mathcal{P}, \mathcal{Q}, \mathcal{S}), (\pr{\mathcal{P}}&, \pr{\mathcal{Q}}, \pr{\mathcal{S}})]_{\HA{n}{2}} \nn \\ 
		&= \Bigl( [\mathcal{P}, \pr{\mathcal{P}}]_{\h}, \frac{1}{2} \bigl( {\coad}_{\H}(\mathcal{P})(\pr{\mathcal{Q}}) - {\coad}_{\H}(\pr{\mathcal{P}})(\mathcal{Q}) \bigr), \bracket{\mathcal{P}}{\pr{\mathcal{Q}}} - \bracket{\mathcal{Q}}{\pr{\mathcal{P}}} \Bigr). \label{eq_Lie_bracket_cal}
	\end{align}
\end{lem}

%

\begin{lem}
\label{lem_technical}
\begin{enumerate}
\item\label{item_lem_technical_dec_PQS} 
Any element $(\cP,\cQ,\cS)$ in $\HG{n}{2}$ can be written as
$$
(\cP,\cQ,\cS)
=
\left(0,\cQ +\frac 14 {\coad}_{\H}(\cP)(\cQ),0\right) 
\odot_{\HG{n}{2}} 
\left(\cP,0,0\right)
\odot_{\HG{n}{2}} 
\left(0,0,\cS+\frac{1}{2} \bracket{\cQ} {\cP})\right).
$$
\item\label{item_lem_technical_ad*_bracket} 
For any $X,X_1,X_2\in \R^{2n+1}$, the following scalar products coincide:
$$
\bracket{{\coad}_{\H}(X)(X_1)} {X_2}
=
\bracket{X_1} {[X_2,X]}.
$$
\item\label{item_lem_technical_master} 
For any $X\in \R^{2n+1}$ and $(\cP,\cQ,\cS)\in \HG{n}{2}$, 
we have
$$
(X,0,0)\odot_{\HG{n}{2}}(\mathcal{P}, \mathcal{Q}, \mathcal{S})
=(0,\cQ',\cS')\odot_{\HG{n}{2}}  (X\HP \cP,0,0) 
$$
for some $\cQ'\in \R^{2n+1}$ 
and $\cS'\in\R$ given by
$$
\cS':=\cS+\bracket{\cQ}{X\HP (\frac{1}{2} \cP)}.
$$
\end{enumerate}
\end{lem}

	\begin{prop} \label{SemiDP}
The Dynin-Folland group can be written as the semi-direct product $\R^{2n+2} \rtimes \H$. The first factor coincides with the normal subgroup $A := \exp(\Lie{a})$, defined by Proposition~\ref{prop}~$(iv)$.
	\end{prop}

\subsection{The Generic Representations of $\HG{n}{2}$} \label{genericrepHG2}

Here we show that the isomorphism $d\pi$ defined in \eqref{def_dpi} can be viewed as the infinitesimal representation of a generic representation $\pi$ of $\HA{n}{2}$.
We will present the argument for the whole family $\pi_\lambda$, $\lambda\in \R\backslash\{0\}$, of generic representations which contains $\pi_1=\pi$.

We begin by defining for each $\lambda\in \R\backslash\{0\}$ the linear mapping 
	\begin{align*}
		d\pi_{\l}: \R^{2(2n+1)+1} \longrightarrow \langle 2 \pi i \hspace{2pt} \scrD_{p_j}, 2 \pi i \hspace{2pt} \scrD_{q_k}, 2 \pi i \hspace{2pt} \scrD_t, 2 \pi i \hspace{2pt} \scrX_{p_l}, 2 \pi i \hspace{2pt} \scrX_{q_m}, 2 \pi i \hspace{2pt} \scrX_{t} \rangle,
	\end{align*}
by
\begin{equation}
\left\{\begin{array}{llll}
d\pi_\lambda(X_{u_j}) = 2 \pi i \hspace{2pt} \scrD_{p_j},&
d\pi_\lambda(X_{v_j})= 2 \pi i \hspace{2pt} \scrD_{q_j},&
d\pi_\lambda(X_w)= 2 \pi i \hspace{2pt} \scrD_t,\\
d\pi_\lambda(X_{x_j})=  \l \hspace{2pt} 2 \pi i \hspace{2pt} \scrX_{p_j},&
d\pi_\lambda(X_{y_j})= \l \hspace{2pt} 2 \pi i \hspace{2pt} \scrX_{q_j},& \\
d\pi_\lambda(X_z)= \l \hspace{2pt} 2 \pi i \hspace{2pt} \scrX_t, 
& d\pi_\lambda(X_s)= \l \hspace{2pt} 2 \pi i \hspace{2pt} \I.\\
\end{array} \right. \label{Lin_Iso_Infites_Rep_Quant_Prob}
\end{equation}
With all our conventions (see Section~\ref{Section_Notation}) we can also write
$$
d\pi_\lambda(u, v, w, x, y, z, s) = 2 \pi i \, \bigl( u \mathscr{D}_{p} + v \mathscr{D}_{q} + w \mathscr{D}_{t} + \l x \mathscr{X}_p + \l y \mathscr{X}_q + \l z \mathscr{X}_t + \l s \I \bigr).
$$

The main property of this subsection is:

\begin{prop} \label{prop_pil}
\begin{enumerate}
\item 
For any $\l\in \R\backslash\{0\}$,
the linear mapping 
$d\pi_\lambda$ is a Lie algebra isomorphism 
between $\HA{n}{2}$ and 
$\langle 2 \pi i \hspace{2pt} \scrD_{p_j}, 2 \pi i \hspace{2pt} \scrD_{q_k}, 2 \pi i \hspace{2pt} \scrD_t, 2 \pi i \hspace{2pt} \scrX_{p_l}, 2 \pi i \hspace{2pt} \scrX_{q_m}, 2 \pi i \hspace{2pt} \scrX_{t} \rangle$.
\item  $d\pi_1=d\pi$.
\item Let $\l\in \R\backslash\{0\}$.
The representation
$d\pi_\lambda$ is the infinitesimal representation of the unitary representation 
$\pi_\lambda$ of $\HG{n}{2}$ acting on $L^2(\H)$
given by
\begin{equation}
		\bigl( \pi_{\l}(\mathcal{P}, \mathcal{Q}, \mathcal{S})f \bigr)(X) = e^{ 2 \pi i \l \left( \mathcal{S} + \bracket{\mathcal{Q}}{X \HP (\frac{1}{2} \mathcal{P})} \right)} \,
		f(X \HP \mathcal{P}), \label{Rep_Neat_Version}
\end{equation}
for $(\cP,\cQ,\cS)\in \HG{n}{2}$, $X\in \H$ and $f\in L^2(\H)$.
\item 
If $\l \not=\l'$  in $\R\backslash\{0\}$, 
the representations $\pi_\lambda$ and $\pi_{\lambda'}$ are inequivalent. 
\end{enumerate}
\end{prop}

\begin{proof}
Parts 1 and 2 are easy to check.

For Part 3, one can check by direct computations that 
\eqref{Rep_Neat_Version}
defines  a unitary representation $\pi_\l$ of $\HG{n}{2}$ and that its infinitesimal representation  coincides with $d\pi_\l$.

Clearly each $\pi_\l$ coincides with the characters 
$\cS \to e^{2 \pi i \l \cS}$ on the centre of the group $\HG{n}{2}$.
Hence, two representations $\pi_\l$ and $\pi_{\l'}$ corresponding 
to different $\l\not=\l'$ are inequivalent,
and Part 4 is proved.
\end{proof}

Let us explain how \eqref{Rep_Neat_Version} appears
by showing that  the unique candidate for the 
representation $\pi_\lambda$ of $\HG n 2$ on $L^2(\H)\cong L^2(\R^{2n+1})$
that admits $d\pi_\lambda$ as infinitesimal representation
is given by \eqref{Rep_Neat_Version}.

As in Proposition~\ref{prop} (iii)
(see also Section~\ref{LVF_Sublap}),
we see that the restriction of  $d\pi_\lambda$  to 
the subalgebra  $\R X_{u_1}\oplus\ldots\oplus \R X_{v_n}\oplus \R X_w \cong \h$
coincides with the infinitesimal right regular representation of $\H$ on $L^2(\R^n)$.
Therefore, the restriction of $\pi_\lambda$ to $\{(\cP,0,0) \mid \cP\in \R^{2n+1}\}$
is given by the right regular representation of $\H$:
$$
		\bigl( \pi_{\l}(\mathcal{P}, 0 , 0) f \bigr)(X) = f(X \HP \mathcal{P}).
$$
The same argument also yields that such $\pi_\lambda$ satisfy 
\begin{eqnarray*}
		\bigl( \pi_{\l}(0, \mathcal{Q}, 0) f \bigr)(X) & = & e^{2 \pi i \l \bracket{\mathcal{Q}} {X} } f(X), \\
		\bigl( \pi_{\l}(0, 0, \mathcal{S}) f \bigr)(X) & = & e^{2 \pi i \lambda \cS} \, f(X). 
\end{eqnarray*}
These equalities  
together with the group law and, more precisely, 
Part \eqref{item_lem_technical_dec_PQS} 
of Lemma\,\ref{lem_technical},
imply 
\begin{eqnarray*}
&&\bigl( \pi_{\l}(\mathcal{P}, \cQ , \cS) f \bigr)(X) 
\\
&&\quad=
\bigl( \pi_\l(0,\cQ+\frac 14 {\coad}_{\H}(\cP)(\cQ),0)\pi_\l (\cP,0,0)\pi_\l (0,0,\cS+\frac{1}{2} \bracket \cQ \cP)  f \bigr)(X)
\\
&&\quad=
e^{2 \pi i \l \bracket{\cQ+\frac 14 {\coad}_{\H}(\cP)(\cQ)}{X}}
 e^{2 \pi i \l (\cS+\frac{1}{2} \bracket\cQ \cP) } 
  f (X \HP \mathcal{P}).
\end{eqnarray*}
By Lemma\,\ref{lem_technical} 
Part \eqref{item_lem_technical_ad*_bracket}, we have
$$
\bracket{{\coad}_{\H}(\cP)(\cQ)}{X}=
\bracket{\cQ}{[X,\cP]},
$$
thus
\begin{align*}
&\bracket{\cQ+\frac 14 {\coad}_{\H}(\cP)(\cQ)}{X}
+\frac{1}{2} \bracket\cQ \cP
=
\bracket{\cQ}{X}
+\frac 14 \bracket{\cQ}{[X,\cP]}
+\frac{1}{2} \bracket\cQ \cP\\
&\qquad
=\bracket{\cQ}{X+\frac{1}{2}\cP +\frac{1}{2} [X,\frac{1}{2}\cP]}
=\bracket{\cQ}{X \HP(\frac{1}{2}\cP) },
\end{align*}
with the convention that the dot product denotes the Heisenberg group law (cf.\,Section~\ref{Section_Notation}).
Hence, the unique candidate for $\pi_\lambda$ is given by \eqref{Rep_Neat_Version}.
Conversely, one checks easily that Formula\,\eqref{Rep_Neat_Version}
defines a unitary representation of $\HG{n}{2}$.

	\begin{cor}
In the exponential coordinates $(u, v, w, x, y, z ,s)$ the representation $\pi_\l$ is given by
	\begin{align}
		\bigl( \pi_{\l}(u, v, w, x, y, z, s)f \bigr)(p, q, t) &= e^{2 \pi i \l \bigl( \bracket{(x, y, z)^t}{(p, q, t)^t} + \frac{1}{2} \bracket{(x, y, z)^t}{(u, v, w)^t} + \frac{1}{4} z \, (pv - uq) + s \bigr)} \nn \\
		& \hspace{10pt} \times f \bigl( p+u, q+v, t+w+\frac{1}{2}(pv-uq) \bigr). \nn 
	\end{align}
	\end{cor}


Since the family of operators
	\begin{align*}
		\bigl( \pi_{\l}(0, \mathcal{Q}, 0) \pi_{\l}(\mathcal{P}, 0 , 0) f \bigr)(X) = e^{2 \pi i \l \bracket{\mathcal{Q}} {X} } f(X \HP \mathcal{P})
	\end{align*}
may be regarded as the natural candidate of time-frequency shifts on $\L{2}{\H}$, we will explore their role more thoroughly from Section~\ref{Section_HM} on.

\section{The Unirreps of $\HG{n}{2}$} \label{Section_Classification_Orbit_Method}

In this subsection we will classify all the unitary irreducible representations of the Dynin-Folland group employing Kirillov's orbit method. 
For a description of this method we refer to, e.g.~\cite{CorwinGreenleaf}.
We will first give a description of the coadjoint orbits of $\HG{n}{2}$. Subsequently, we will construct the corresponding unirreps. Finally, for each orbit we will have a look at the corresponding jump sets.

\subsection{The Coadjoint Orbits}

In order to classify the $\HG{n}{2}$-coadjoint orbits, we first we give an explicit formula for the coadjoint representation ${\coAd}_{\HG{n}{2}}$ of $\HG{n}{2}$ 
on the dual $\HA{n}{2}^*$  of its Lie algebra $\HA n 2$.
Recall that ${\coAd}_{\HG{n}{2}}$ is defined by
\begin{equation}
\label{eq_def_Kg}
		\bracket{{\coAd}_{\HG{n}{2}}(g)F}{X} 
		= \bracket{F}{{\Ad}_{\HG{n}{2}}(g^{-1}){X}}
\end{equation}
for $g\in \HG{n}{2}$, $X\in \HA{n}{2}$ and $F \in \HA{n}{2}^*$.
We denote by
$$
(X_{u_1}^*, \ldots, X_{u_n}^*, X_{v_1}^*, \ldots, X_{v_n}^*, X_w^*, X_{x_1}^*, \ldots, X_{x_n}^*, X_{y_1}^*, \ldots, X_{y_n}^*, X_z^*, X_s^*),
$$
the dual standard basis of $\R^{2(2n+1)+1}$. We check readily (cf.~\cite[Lem~3.11]{Ro14}):

\begin{lem}
\label{lem_K}
For any $X\in \HA{n}{2}$ and  $F \in \HA{n}{2}^*$
with
	\begin{align*}
		F &= f_u X^*_u + f_v X^*_v + f_w X^*_w + f_x X^*_x + f_y X^*_y + f_z X^*_z + f_s X^*_s, \\
		X &= u X_u + v X_v + w X_w + x X_x + y X_y + z X_z + s X_s,
\end{align*}
we have
\begin{eqnarray*}
		{\coAd}_{\HG{n}{2}}(\exp_{\HG{n}{2}}(X))F 
		&=& 
		\bigl( f_u + f_w v - \frac z2 f_y +  f_s x + \frac{3}{4} f_s z v \bigr) X^*_u \\
		&&\ +
		\bigl( f_v - f_w u + \frac z2 f_x +  f_s y - \frac{3}{4} f_s z u \bigr) X^*_v \\
		&&\ +
		\bigl( f_w + f_s z \bigr) X^*_w 
		\quad +
		\bigl( f_x -  f_s u \bigr) X^*_x 
		\quad +		
		\bigl( f_y - f_s v \bigr) X^*_y \\
		&&\ +
		\bigl( f_z - \frac{f_x v}2 + \frac{f_y u}2 - f_s w \bigr) X^*_z 
		\quad +
		 f_s X^*_s.
\end{eqnarray*}
\end{lem}


The specific form of the coadjoint action 
implies that the coadjoint orbits of $\HG{n}{2}$ are all affine subspaces. This was already indicated in \cite{FollMeta}. In the following we provide an explicit description of the orbits
by giving their representatives. 
Given our convention, 
we may write 
$\R^n X_x$ for $\R X_{x_1}\oplus \ldots \oplus \R X_{x_n}$
and similarly for
$\R^n X_y$, $\R^n X_u$, $\R^n X_v$ etc.

\begin{prop} \label{prop_coadjoint_orbit}
Every coadjoint orbit of $\HG{n}{2}$ 
has exactly one representative among the following elements of $\HA{n}{2}^*$:
\begin{itemize}
\item[(Case (1))] 
 $f_s X^*_s$ if $f_s\not=0$.
\item[(Case (2))] 
$f_w X^*_w+f_x X_x^* +f_y X_y^*+
f_z X_z^*$
if $f_s=0$ but $f_w\not=0$.
\item[(Case (3))]
$f_u  X^*_u +f_v X^*_v +	 f_x X^*_x +  f_y X^*_y$
if $f_u f_y = f_v f_x$, $f_s = f_w = 0$ and $(f_x,f_y) \neq 0 \in \R^{2n}$.
\item[(Case (4))]
$f_u  X^*_u +f_v X^*_v +f_z X^*_z$
if $f_s = f_w=0$, $f_x=f_y=0$.
\end{itemize}

All coadjoint orbits except the singletons are affine subspaces of $\HA{n}{2}^*$.
More precisely, 
in Case (1), the orbit of $f_s X^*_s$ 
is the affine hyperplane passing through $f_s X^*_s$ given by
	\begin{align}
		{\coAd}_{\HG{n}{2}}(\HG{n}{2} )(f_s X^*_s) = f_s X^*_s \oplus \R^n X^*_{u}\oplus \R^n X^*_{v} \oplus \R X^*_w \oplus \R^n X_{x}\oplus  \R^n X^*_{y} \oplus \R X^*_z. \label{OrbitCase1}
	\end{align}
The orbits ${\coAd}_{\HG{n}{2}}(\HG{n}{2} )(f_s X^*_s)$, $f_s \in \R\backslash\{0\}$,
are the generic coadjoint orbits.
They form an open dense subset of $\HA2n^*$.

In Case (2), the orbits are the $2n$-dimensional affine subspaces
\begin{align}
&{\coAd}_{\HG{n}{2}}(\HG{n}{2} )(f_w X^*_w+f_x X_x^* +f_y X_y^*+
f_z X_z^*) \nn
\\&\ =
f_w X^*_w+f_x X_x^* +f_y X_y^*+
f_z X_z^*+
\{\tilde v X_u^* + \tilde u X_v^* -\frac{f_x \tilde{v} + f_y \tilde{u}}{2f_w} X_z^* \mid
\tilde u,\tilde v\in \R^n\}. \label{OrbitCase2}
\end{align}

In Case (3), the orbits are the $2$-dimensional affine subspaces
\begin{align}
&{\coAd}_{\HG{n}{2}}(\HG{n}{2} )(f_u  X^*_u +f_v X^*_v +	 f_x X^*_x +  f_y X^*_y) \nn
\\&\quad=
f_u  X^*_u +f_v X^*_v +	 f_x X^*_x +  f_y X^*_y
+ \R (-f_y X_u^* +f_x X_v^*) +\R X_z^*. \label{OrbitCase3}
\end{align}

In Case (4), the orbits are the singletons $f_u X^*_u + f_v X^*_v + f_z X^*_z$ with $f_u, f_v \in \R^n, f_z \in \R$.
\end{prop}

\begin{proof}
Case $(1)$ Let $F \in \HA{n}{2}^* \setminus \{ 0 \}$ with non-vanishing component $f_s$.
Then we choose $X$ as in Lemma\,\ref{lem_K} 
with 
$z,u,v$ such that the coordinates 
of ${\coAd}_{\HG{n}{2}}(\exp_{\HG{n}{2}}  X)(F)$
in $X^*_w$, $X_x^*$ and $X^*_y$ are zero, that is,
$$
f_w+f_s z= 0,\quad f_x-f_s u=f_y- f_sv=0,
$$
and we choose $w,x,y$ such that
 the coordinates in $X^*_z$, $X_u^*$ and $X^*_v$ are zero, that is,
$$
f_z - \frac{f_x v}2 + \frac{f_y u}2 - f_s w =0,
$$
and 
$$
0\,=\, f_u + f_w v  - \frac z2 f_y + f_s x + \frac{3}{4} f_s z v
\, = \,
f_v - f_w u + \frac z2 f_x  + f_s y - \frac{3}{4}   f_s z u
.
$$
Thus, we have obtained  ${\coAd}_{\HG{n}{2}}(\exp(X))F = f_s X^*_s$;
equivalently, the orbit ${\coAd}_{\HG{n}{2}}(\HG{n}{2})F$ describes the $2 (2n + 1)$-dimensional hyperplane through $f_s X^*_s$ which is parallel to the subspace ${\HA{n}{2}}^* / {\R X^*_s}$.

Case $(2)$. We assume $f_s=0$ but $f_w\not=0$, so that we have
\begin{eqnarray*}
		{\coAd}_{\HG{n}{2}}(\exp_{\HG{n}{2}}(X))F 
		&=&\bigl( f_u + f_w v - \frac12 \, z f_y \bigr) X^*_u 
+\bigl( f_v -  f_w u + \frac12 \, z f_x \bigr) X^*_v \\
&&\quad		+ f_w X^*_w 
		+  f_x X^*_x +  f_y X^*_y 
		+\bigl( f_z - \frac12 \, f_x v + \frac12 \, f_y u \bigr) X^*_z.
\end{eqnarray*}
We choose $u$ and $v$ such that the coordinates in $X^*_u$
and $X^*_v$ vanish, that is,
$$
v := \frac 1{f_w}(-f_u+\frac{1}{2} zf_y)
\quad\mbox{and}\quad
u := \frac 1{f_w}(f_v+\frac{1}{2} zf_x).
$$
Then the $X_z^*$-coordinate of ${\coAd}_{\HG{n}{2}}(\exp_{\HG{n}{2}}(X))F$ becomes
$$
f_z - \frac12 \, f_x v + \frac12 \, f_y u
=
f_z +\frac 1{2f_w}(f_uf_x+f_vf_y)
$$
independently of the other entries $w,x,y,z,s$ of $X$.
Therefore, 
$F':=f_w X^*_w+f_x X_x^* +f_y X_y^*+ f'_z X_z^*$
with $f'_z := f_z +\frac 1{2f_w}(f_uf_x+f_vf_y)$
is in the same orbit as $F$
and $F'$ is the only element of the orbit with zero coordinates in $X_u^*$ and $X_v^*$.
We choose $F'$ as the representative 
of the coadjoint orbit that contains $F$.
Similar  computations as above, 
together with setting $\tilde v := f_w v - \frac z 2  f_y\in \R^n$ and 
$\tilde u := -  f_w u + \frac z2  f_x\in \R^n$, yield
\begin{eqnarray*}
{\coAd}_{\HG{n}{2}}(\exp_{\HG{n}{2}}(X))F'
&=& 
F'+
\tilde v  X^*_u + \tilde u  X^*_v 
-\frac{f_x  \tilde{v} + f_y \tilde{u}}{2 f_w}  X^*_z.
\end{eqnarray*}
This yields the description of the $F'$-orbit.

Case (3). We assume $f_s= 0 =f_w$.
Then 
\begin{eqnarray*}
		{\coAd}_{\HG{n}{2}}(\exp_{\HG{n}{2}}(X))F 
		&=&\bigl( f_u  - \frac z2  f_y \bigr) X^*_u 
+\bigl( f_v+ \frac z2  f_x \bigr) X^*_v \\
&&\quad			+  f_x X^*_x +  f_y X^*_y 
		+\bigl( f_z - \frac12  f_x v + \frac12  f_y u \bigr) X^*_z.
\end{eqnarray*}
We also assume $(f_x,f_y)\not=0$.
Then we can choose $v$ or $u$ such that the $X^*_z$-coordinate vanishes, and
for $z = 0$ the $u$- and $v$-coordinates $f_u X^*_u$ and $f_v X^*_v$,
respectively.
This means that, in this case, 
$F$ and $F':= f'_u  X^*_u +f'_v X^*_v +	 f_x X^*_x +  f_y X^*_y$
with $(f'_u,f'_v) \perp (-f_y,f_x)$ in $\R^{2n}$ are in the same orbit.
Furthermore, $F'$ is the only element of this orbit 
with $(f'_u,f'_v) \perp (-f_y,f_x)$.
Similar computations as above,
with $\tilde z := \frac z2 \in \R$
and
$\tilde a: = f_z - \frac12 \, f_x v + \frac12 \, f_y u\in \R$, give
\begin{eqnarray*}
{\coAd}_{\HG{n}{2}}(\exp_{\HG{n}{2}}(X))F'
&=& 
F' - \tilde z f_y X^*_u + \tilde z f_x X^*_v
+\tilde a X^*_z.
\end{eqnarray*}
This yields the description of the $F'$-orbit.

If $f_x=f_y=0$, then 
$F= f_u  X^*_u +f_v X^*_v +	 f_z X^*_z
={\coAd}_{\HG{n}{2}}(\exp_{\HG{n}{2}}(X))F$ for any $X\in \HA{n}{2}$.
This corresponds to Case $(4)$.
This concludes the proof of Proposition~\ref{prop_coadjoint_orbit}.
\end{proof}

\subsection{The Unirreps}

To begin with, let us show that the representations corresponding to the orbits of Case $(1)$
via the orbit method
coincide with the representations $\pi_\l$
constructed in Section~\ref{genericrepHG2}:
 
\begin{prop}
\label{prop_pil_F}
Let $f_s = \l \in \R \setminus \{ 0 \}$.
The representation $\pi_\l$  
as defined by \eqref{Rep_Neat_Version}
is unitarily equivalent to 
the unirrep corresponding to the linear form 
$F_1 := \l X^*_s$. A
maximal (polarising) subalgebra subordinate to $F_1$ is
$$
\Lie{l} := \Lie{a} = \R^n X_{x}  \oplus  \R^n X_{y}  \oplus  \R X_z  \oplus  \R X_s.
$$
\end{prop}

\begin{proof}
One checks easily that the subspace $\Lie l$ of $\HA{n}{2}$ 
is a maximal subalgebra subordinate to $F_1$
and that its corresponding subgroup is 
$$
\mathbf{L} = \exp_{\HG{n}{2}} (\Lie l) =\{(0,\cQ,\cS) \mid \cQ\in \R^{2n+1}, \ \cS\in \R\}.
$$
Let  $\rho_{F_1, \mathbf{L}}$ be the character of the subgroup 
$\mathbf{L}$
with infinitesimal character $2 \pi i F_1$.
It is given for any $X=x X_x + y X_y + z X_z + s X_s\in \Lie l$ by
$$
\rho_{F_1,\bfL}(\exp_{\HG{n}{2}} (X))= e^{2 \pi i F_1(X)} =e^{2 \pi i \lambda s},
$$
and also for any $(0,\cQ,\cS) \in \bfL$ by 
\begin{equation}
\label{eq_rho_cS}
\rho_{F_1,\bfL}(0,\cQ, \cS )= e^{2 \pi i \lambda \cS }.
\end{equation}

In order to define the representation  induced by $\rho_{F_1,\bfL}$,
we consider $\cF_0$, the space of continuous
functions $\varphi:\HG{n}{2} \to \CF$
that satisfy
\begin{equation}
\label{eq_def_varphi_in_cF0}
\varphi( \ell \odot_{\HG{n}{2}} g) = \rho_{F_1, \mathbf{L}}(
\ell) \, \varphi(g),\quad
		 \mbox{for all } \ell\in \bfL,\  g \in \HG{n}{2}, 
\end{equation}
and whose support modulo $\bfL$ is compact.
Let $\MOP{ind}(\rho_{F_1, \mathbf{L}})^{\HG{n}{2}}_\mathbf{L}$ 
be the representation induced by $\rho_{F_1, \mathbf{L}}$ on the group 
$\HG{n}{2}$ that acts on $\cF_0$.
It may be realised as 
$$
		\Bigl( \bigl( \MOP{ind}(\rho_{F_1, \mathbf{L}})^{\HG{n}{2}}_\mathbf{L}(g) \bigr) \varphi \Bigr)(g_1) := \varphi(g_1 \odot_{\HG{n}{2}} g),\quad
		g,g_1\in \HG{n}{2}, \ \varphi\in \cF_0.
		$$

By Proposition~\ref{prop_HG2n}, 
 the subset $\{(X,0,0) \mid X\in \R^{2n+1}\}$ 
of $\HG n 2$ is a subgroup of $\HG{n}{2}$ 
which is isomorphic to 
the Heisenberg group $\H$.
Here, we allow ourselves to identify this subgroup with $\H$.
Let $U$ denote the restriction map from $\HG{n}{2}$ to $\H$,
that is, 
$$
U(\varphi)(X) =\varphi (X,0,0)
$$
for any scalar function $\varphi:\HG{n}{2} \to \CF$.
Clearly, if $\varphi\in \cF_0$, 
then $U\varphi$ is in $C_c(\H)$, 
the space of continuous functions with compact support on $\H$. 
In fact, a function $\varphi\in \cF_0$ 
is completely determined by its restriction to $\H$
since the Lie algebra of $\H$ within $\HA{n}{2}$ complements $\Lie l$.
With this observation 
it is easy to check that  
$U$ is a linear isomorphism from $\cF_0$ to $C_c(\H)$.
Since $C_c(\H)$ is dense in the Hilbert space $L^2(\H)$,
the proof will be complete once we have  shown that 
the induced representation 
$\MOP{ind}(\rho_{F_1, \mathbf{L}})^{\HG{n}{2}}_\mathbf{L}$
intertwined with $U$
coincides with the representation $\pi_\lambda$ acting on $C_c(\H)$,
that is,
\begin{equation}
\label{eq_pf_prop_pil_F_main}
\forall g\in \HG{n}{2},\
\forall \varphi\in \cF_0
\qquad
U \left[\MOP{ind}(\rho_{F_1, \mathbf{L}})^{\HG{n}{2}}_\mathbf{L} (g) (\varphi)\right]
= \pi_\lambda(g) (U \varphi).
\end{equation}

Let us prove \eqref{eq_pf_prop_pil_F_main}.
We fix a function $\varphi\in \cF_0$.
By Lemma\,\ref{lem_technical} Part \eqref{item_lem_technical_master},
we have for $g=(\mathcal{P}, \mathcal{Q}, \mathcal{S})$
and $g_1=(X,0,0)\in \H$
\begin{align*}
		\Bigl( \bigl( \MOP{ind}(\rho_{F_1, \mathbf{L}})^{\HG{n}{2}}_\mathbf{L}(g) \bigr) \varphi \Bigr)(X, 0, 0) 
&= \varphi\left((X, 0, 0) \odot_{\HG{n}{2}} (\mathcal{P}, \mathcal{Q}, \mathcal{S})\right) 
\\
&= \varphi\left(\ell \odot_{\HG{n}{2}} (X\HP \mathcal{P}, 0, 0) \right)
\end{align*}
with 
$$
\ell = 
\left(0, \cQ', \cS +\bracket{\cQ}{X\HP (\frac{1}{2}\cP)} \right)  \in \bfL,
$$
for some $\cQ'\in \R^{2n+1}$.
Since $\varphi$ is in $\cF_0$, it satisfies \eqref{eq_def_varphi_in_cF0}
and we have
\begin{eqnarray*}
\varphi\left(\ell \odot_{\HG{n}{2}} (X\HP \mathcal{P}, 0, 0) \right)
&=&
\rho_{F_1, \mathbf{L}}(\ell) \, \varphi(X\HP \mathcal{P}, 0, 0) 
\\
&=&
e^{2 \pi i \l (\cS +\bracket{\cQ}{X\HP (\frac{1}{2}\cP)} )}
 \varphi(X\HP \mathcal{P}, 0, 0)
\end{eqnarray*}
by \eqref{eq_rho_cS}.
We recognise $\bigl( \pi_\l(g) f \bigr) (X)$ with $f=U\varphi$
due to \eqref{Rep_Neat_Version}.
Therefore, \eqref{eq_pf_prop_pil_F_main} is proved, 
and the proof is complete.
\end{proof}

Proceeding as in Proposition~\ref{prop_pil_F}, we can give concrete realisations in $L^2(\R^n)$ 
and $L^2(\R)$
of the unirreps associated with the coadjoint orbits 
of Cases (2) and (3) in Proposition~\ref{prop_coadjoint_orbit} (for a proof see \cite[Prop.~3.15]{Ro14}):

\begin{prop}
\label{prop_Case2+3}
$\bullet$ {\rm (Case (2))}
Let $F_2:=f_w X^*_w+f_x X_x^* +f_y X_y^*+
f_z X_z^* \in \HA{n}{2}^*$
with  $f_s=0$ but $f_w\not=0$.
A maximal (polarising) subalgebra subordinate to $F_2$ is 
$$
\Lie{l}_2 := 
\R^n X_v
\oplus \R X_w\oplus
\R^n X_{x} \oplus  \R^n X_{y} 
\oplus
\left\{\frac z{2f_w} f_x X_u +zX_z \mid  z\in \R\right\}
 \oplus  \R X_s.
$$
The associated unirrep of $\HG{n}{2}$ may be realised as 
the representation $\pi_{(f_w,f_x,f_y,f_z)}$ 
acting unitarily on $L^2(\R^n)$ via
\begin{eqnarray*}
\left(\pi_{(f_w,f_x,f_y,f_z)}(g) \psi\right) (\tilde u) 
&=& 
\psi(\tilde u+u -\frac z {2f_w} f_x)
e^{2 \pi i (wf_w+xf_x +yf_y +zf_z)}\\
&&
\qquad \exp \pi i
{\bracket  {2\tilde u+u -\frac z {2f_w} f_x} {f_w v - \frac z2  f_y}}_{\R^n}
\end{eqnarray*}
for $g=(u,v,w,x,y,z,s) \in \HG{n}{2}$, $\psi \in L^2(\R^n)$, 
and $\tilde u\in \R^n$.

$\bullet$ {\rm (Case (3))}
Let $F_3 := f_u  X^*_u +f_v X^*_v +	 f_x X^*_x +  f_y X^*_y\in \HA{n}{2}^*$
with $f_u f_y =  f_vf_x $
and $f_s=f_w=f_z=0$ but $(f_x,f_y)\not=0$.
A maximal (polarising) subalgebra subordinate to $F_3$ is 
$$
\Lie{l}_3 := 
\R^n X_u \oplus \R^n X_v \oplus \R X_w\oplus
\R^n X_{x} \oplus  \R^n X_{y}   \oplus  \R X_s. 
$$
The associated unirrep of $\HG{n}{2}$ may be realised as 
the representation $\pi_{(f_u,f_v,f_x,f_y)}$ 
acting unitarily on $L^2(\R)$ by
\begin{eqnarray*}
\left(\pi_{(f_u,f_v,f_x,f_y)}(g) \psi\right) (\tilde z) 
&=& 
\psi (\tilde z+z) e^{2 \pi i(f_u u + f_v v + f_x x + f_y y)}
\\&&\qquad
\exp \pi i \Bigl(    
\frac{2\tilde z+z}{2} (-f_x v + f_y u)
\Bigr), 
\end{eqnarray*}
for $g=(u,v,w,x,y,z,s) \in \HG{n}{2}$, $\psi\in L^2(\R)$, 
and $\tilde z\in \R$.
\end{prop}

By Kirillov's orbit method \cite{Kir,CorwinGreenleaf}, 
Propositions \ref{prop_coadjoint_orbit}, \ref{prop_pil_F},
and \ref{prop_Case2+3}
imply the following classification of the unitary dual 
of the Dynin-Folland group: 

\begin{thm} \label{Thm_Classification_Unirreps}
Any unitary irreducible representation of the Dynin-Folland group $\HG{n}{2}$ is unitarily equivalent to exactly one of the following
representations:
\begin{itemize}
\item[(Case (1))]  $\pi_\lambda$ for $\lambda\in \R\backslash\{0\}$, 
acting on $L^2(\H)$,
defined in Proposition~\ref{prop_pil},
\item[(Case (2))]  
$\pi_{(f_w,f_x,f_y,f_z)}$
for any $f_x,f_y\in \R^n$, $f_z\in \R$ and $f_w\in \R\backslash\{0\}$,
acting on $L^2(\R^n)$,
defined in Proposition~\ref{prop_Case2+3},
\item[(Case (3))] 
$\pi_{(f_u,f_v,f_x,f_y)}$
for any $f_u,f_v,f_x,f_y\in \R^n$
with $f_u f_y= f_v f_x $
 but $(f_x,f_y)\not=0$,
 acting on $L^2(\R)$,
defined in Proposition~\ref{prop_Case2+3},
\item[(Case (4))] 
the characters $\pi_{(f_u,f_v,f_z)}$
given for any $f_u,f_v\in \R^n$ and $f_z\in\R$ by
$$
\pi_{(f_u,f_v,f_z)}(u,\ldots,s)= e^{2 \pi i (uf_u +vf_v+zf_z)},\qquad (u,\ldots,s)\in \HG{n}{2}.
$$
\end{itemize}
\end{thm}

\subsection{Projective Representations for  $\HG{n}{2}$} \label{Subsection_SI}

In this subsection, we give explicit descriptions of the projective representations of  $\HG{n}{2}$.

We first need to discuss the square-integrability of the unirreps of $\HG{n}{2}$ and to describe their projective kernels.
Recall that the projective kernel of a representation is, by definition, the smallest connected subgroup containing the kernel of the representation.

\begin{thm} \label{ProjKer}
\begin{itemize}
\item[(i)]
Every coajoint orbit of $\HG{n}{2}$ is flat in the sense that it is an affine subspace of $\HA{n}{2}^*$.
Consequently,
every unirrep $\pi$ of $\HG{n}{2}$  is square integrable modulo its projective kernel.
Moreover, the projective kernel of $\pi$ 
coincides with the stabiliser subgroup (for the coadjoint action of $\HG{n}{2}$) of a linear form corresponding to $\pi$.
\item[(ii)]
The respective projective kernels are the following subgroups  of $\HG{n}{2}$:
	\begin{itemize}
		\item[(Case (1))]  
		$R_{F_1} = \exp_{\HG{n}{2}}(\R X_s) = Z(\HG{n}{2})$
		for the representation $\pi_\lambda$,
		\item[(Case (2))] 		$R_{F_2} = \exp_{\HG{n}{2}}\bigl( \R X_w \oplus \R^n X_x \oplus \R^n X_y \oplus \R X_z \oplus \R X_s \bigr)$
		for the representation  $\pi_{(f_w,f_x,f_y,f_z)}$,
		\item[(Case (3))]  $R_{F_3} = \{ (u, v, w, x, y, z ,s) \in \HG{n}{2} \mid z = 0, f_x v= f_y u \}$ for the representation  $\pi_{(f_u,f_v,f_x,f_y)}$,
		\item[(Case (4))] $R_{F_4} = \HG{n}{2}$ for the character $\pi_{(f_u,f_v,f_z)}$.
	\end{itemize}
	\end{itemize}
\end{thm}

In Part (ii) above, we used the notation of Propositions~\ref{prop_pil_F} and \ref{prop_Case2+3} and  Theorem~\ref{Thm_Classification_Unirreps}.

\begin{proof}
In	Proposition \ref{prop_coadjoint_orbit}, we already observed that the coadjoint orbits were affine subspaces.
So Part (i) follows from well-known properties of square integrable unirreps of nilpotent Lie groups, see Theorems~3.2.3 and 4.5.2 in \cite{CorwinGreenleaf}.
Part (ii) is proved by computing  the stabiliser of each linear form $F_1,\ldots,F_4$ using  Lemma~\ref{lem_K}.
\end{proof}

Given a unirrep $\pi$ 
which is square integrable modulo its projective kernel, 
 it is convenient to quotient it by the projective kernel, in which case we speak of a projective representation and denote it by $\pi^{\mathrm{pr}}$.
We conclude this subsection by listing the projective unirreps of the $\HG{n}{2}$, which will give us useful insight.

	\begin{cor} \label{CorProjReps}
$\bullet$ {\rm (Case (1))} Let $\l \in \R \setminus \{ 0 \}$. Then the projective representation of $\HG{n}{2}/R_{F_1}$ corresponding to $\pi_\lambda$ is given by
	\begin{align*}
		\bigl( \pi^{\rm{pr}}_{\l}(\mathcal{P}, \mathcal{Q})\psi \bigr)(X) = e^{ 2 \pi i \l \bracket{\mathcal{Q}}{X \HP (\frac{1}{2} \mathcal{P})}} \hspace{2pt}
		\psi(X \HP \mathcal{P}).
	\end{align*}
$\bullet$ {\rm (Case (2))} Let $f_x,f_y\in \R^n$, $f_z\in \R$ and $f_w\in \R\backslash\{0\}$. Then the projective representation of $\HG{n}{2}/R_{F_2}$ corresponding to $\pi_{(f_w,f_x,f_y,f_z)}$ is given by
	\begin{align*}
		\left(\pi^{\rm{pr}}_{(f_w,f_x,f_y,f_z)}(u, v) \psi\right) (\tilde u) = e^{2 \pi i f_w \bigl( \tilde{u} v + \frac{u v}{2} \bigr)} \hspace{2pt} \psi(\tilde u+u).
	\end{align*}
In particular, it is unitarily equivalent to the projective Schr\"{o}dinger representation of $\H/Z(\H)$ of Planck's constant $f_w$.

$\bullet$ {\rm (Case (3))} Given $f_x, f_y, f_x, f_y \in \R^{n}$ with $(f_x, f_y) \neq 0 \in \R^{2n}$ and $\bracket{(-f_y, f_x)}{(f_u, f_v)} = 0$, set $\l_{f_x, f_y} X_\zeta^* := -f_x X^*_v + f_y X^*_u$ for $\l_{f_x, f_y} := \Abs{(f_x, f_y)}$. Then the projective representation of $\HG{n}{2}/R_{F_3}$ corresponding to 
$ \pi_{(f_u,f_v,f_x,f_y)}$
is given by
	\begin{align*}
		\left(\pi^{\rm{pr}}_{(f_u,f_v,f_x,f_y)}(z, \zeta) \psi \right) (\tilde{z}) :=& \left(\pi^{\rm{pr}}_{(f_u,f_v,f_x,f_y)}(\exp(z X_z + \zeta X_\zeta) \psi \right) (\tilde{z} X_z) \\
		=& e^{2 \pi i \l_{f_x, f_y} \bigl( \tilde{z} \zeta + \frac{z \zeta}{2} \bigr)} \hspace{2pt} \psi(\tilde{z} + z).
	\end{align*}
In particular, it is unitarily equivalent to the projective Schr\"{o}dinger representation of $\mathbf{H}_1/Z(\mathbf{H}_1)$ of Planck's constant $\l_{f_x, f_y}$.

$\bullet$ {\rm (Case (4))} is trivial.
	\end{cor}

\section{The Heisenberg-Modulation Spaces $\E{\pt, \qt}{}{n}$} \label{Section_HM}

In this section we introduce the notion of Heisenberg-modulation spaces, a new class of function spaces on $\R^{2n+1}$. In analogy to the definition of the classical modulation spaces $\M{\pt, \qt}{\vt}{n}$, we employ a specific type of generalised time-frequency shifts, realised in terms of the generic representation $\pi = \pi_1$ of the Dynin-Folland group $\HG{n}{2}$ discussed in Subsection~\ref{genericrepHG2}. Our approach is motivated by the following theoretical question: are there any families of generalised time-frequency shifts
 such that the coorbit spaces they induce differ from the modulation spaces $\M{\pt, \qt}{\vt}{n}$? 
The rest of the paper is devoted to showing that the answer is yes 
by constructing the spaces associated with the
 projective representations of the   Dynin-Folland groups constructed in the previous section.

This section gives a first definition of the spaces and sets the stage for a thorough analysis and proof of novelty. The first subsection merges the coorbit and decomposition space pictures of classical modulation spaces and provides a guideline for the analysis. Based on this guideline, the second subsection offers a first definition of the Heisenberg-modulation spaces.

\medskip

\noindent\textbf{Convention.}
Throughout the paper the bracket $\subbracket{\SD{n}}{\, . \,}{\, . \,}$ denotes the sesqui-linear $\SD{n}$-$\SF{n}$-duality which extends the natural inner product on $\L{2}{\R^n}$.

%
%

\subsection{Merging Perspectives} \label{Merge}

In this subsection we briefly review some aspects of classical modulation spaces and homogeneous Besov spaces on $\R^n$.
For the necessary background in coorbit theory and decomposition space theory, we refer to Feichtinger and Gr\"{o}chenig's foundational paper~\cite{fegr89} on coorbit theory and to Feichtinger and Gr\"{o}bner's foundational paper \cite{fegr85} on decomposition space theory.

\subsubsection{Modulation spaces on $\R^{n}$ defined via the Schr\"odinger representation}

It is well known that given $\pt, \qt \in [1, \infty)$, a weight of polynomial growth $\vt$ on $\R^{2n}$ and any non-vanishing window $\psi \in \SF{n}$, the modulation space $\M{\pt, \qt}{\vt}{n}$ can be defined as the space of all $f \in \SD{n}$ such that
	\begin{align}
		\Norm{\M{\pt, \qt}{\vt}{n}}{f}
		&= \Bigl(  \int_{\R^n} \Bigl( \int_{\R^n} \vt(q,p)^\pt \Abs{\int_{\R^n} f(x) \hspace{2pt} e^{-2 \pi i px} \hspace{2pt} \overline{\psi}(x+q) \, dx}^\pt \, dq \Bigr)^{\qt/\pt} \, dp \Bigr)^{1/\qt} < \infty. \label{SplitExpSchrod}
	\end{align}
The spaces do not depend on the choice of window $\psi\in\SF{n} \setminus \{ 0 \}$ and 	
the typical examples are $\vt_\st(p, q) = (1 + \Abs{p}^2 + \Abs{q}^2)^{\st/2}$ and $\vt_\st(p) = (1 + \Abs{p}^2)^{\st/2}$. Cf.~Gr\"{o}chenig~\cite[\SS~11.1]{gr01}.

A reformulation of the norm \eqref{SplitExpSchrod} in terms of the projective Schr\"{o}dinger representation of $\H$ and the mixed-norm Lebesgue space $L^{\pt, \qt}_\vt$ over $\H/Z(\H) \cong \R^{2n}$ gives access to viewpoints of both coorbit theory and decomposition space theory. 
We realise the Heisenberg group $\H$ as the polarised Heisenberg group (see e.g. \cite[p.~68]{FollPhSp}), by employing coordinates we will call split exponential coordinates for $\H$.
This means writing a generic element of $\H$ as 
$$
\bigl( (q, t), p \bigr) :=  (0, q, t) (p, 0, 0) = (p,q,t +\frac 12 qp),
$$
so that, with this notation,
the group law of $\H$ is given by
$$
\bigl( (q, t), p \bigr)\bigl( (q', t'), p' \bigr)
=
\bigl( (q+q', t+t'+pq'), p+p' \bigr).
$$
The structure of $\H$ as  the semi-direct product 
 $\H= \R^{n+1} \rtimes \R^n$ is now more apparent.
Accordingly, we equip $\H$ with the bi-invariant Haar meausure $\mu_{\H}$ in split exponential coordinates. This is precisely the Lebesgue measure $d(q,t) \,dp$, which coincides with $dt \,dq \,dp$ since $Z(\H) = \exp(\R X_t)$ is normal in $\exp(\R^n X_q \oplus \R X_t)$.

The representation of $\H$ induced from the character 
$$
\chi: \exp(\R Xq + \R X_t) \subgr \H \to \CF: \bigl( (q, t), 0 \bigr) \mapsto e^{2 \pi i t},
$$ gives the Schr\"odinger representation defined in Section  \ref{Unirreps_H1n}.
In split exponential coordinates, it is given by 	
\begin{align*}
		\Bigl(\rho\bigl( (q,t), p \bigr) \psi \Bigr)(x) = e^{2 \pi i (t + q x)} \hspace{2pt} \psi(x + p).
	\end{align*}
	Quotienting by the centre $Z(\H) = \exp(\R X_t)$,
this yields the projective representation
	\begin{align*}
		\Bigl( \rho^{pr}(q, p) \psi \Bigr)(x) = e^{2 \pi i q x} \hspace{2pt} \psi(x + p) = (M_q T_p \psi)(x)
	\end{align*}
of $\H /Z(\H)$, respectively.
Note that the quotient of the Haar measure by the centre equals the measure $dq \,dp$.

We notice that these are the representation and measure used in \eqref{SplitExpSchrod} except for the exchanged roles of $p$ and $q$. In fact, the projective representation used in \eqref{SplitExpSchrod}  is not $\rho^{pr}$ but is related to it via intertwining with the Fourier transform on $\R^n$. Indeed, setting
\begin{equation}
\label{FourierInvSchrod}
\widecheck{\rho}^{\mathrm{pr}} := \F \rho^{\mathrm{pr}} \F^{-1},
\end{equation}
 a short calculation yields
$$
		\Bigl(\widecheck{\rho}^{\mathrm{pr}} (q, p) \psi \Bigr)(x) = \Bigl( \F \bigl(  \rho^{\mathrm{pr}}  (q, p) \F^{-1} \psi \bigr)\Bigr)(x) = e^{2 \pi i p q} \hspace{2pt} e^{2 \pi i px} \hspace{2pt} \psi(x+q). 
$$
Therefore, we can view $\M{\pt, \qt}{\vt}{n}$ as the space of all $f \in \SD{n}$ such that
	\begin{align}
		\Norm{\M{\pt, \qt}{\vt}{n}}{f}& = \Bigl(  \int_{\R^n} \Bigl( \int_{\R^n} \vt(q,p)^\pt \Abs{\subbracket{\L{2}{\R^n}}{f}{\widecheck{\rho}^{\mathrm{pr}} (q, p) \psi}}^\pt \, dq \Bigr)^{\qt/\pt} \, dp \Bigr)^{1/\qt} 		< \infty.  \label{NormSplitExpSchrod}
	\end{align}
This is precisely the definition of modulation spaces in the framework of coorbit theory (cf.~Feichtinger and Gr\"{o}chenig's foundational papers \cite{FeiGroe86, fegr89}) once it is observed that $\SF{n}$ is a good enough reservoir of test functions and that the Schr\"{o}dinger representation of the reduced Heisenberg group can be replaced by the projective Schr\"{o}dinger representation (cf.~Christensen~\cite{Chr96} and Dahlke~et~al.~\cite{dastte04-2}).

We can also rewrite \eqref{NormSplitExpSchrod} in terms of mixed Lebesgue norms as
\begin{equation}
\Norm{\M{\pt, \qt}{\vt}{n}}{f} =
\Norm{\LW{\pt, \qt}{\H/Z(\H)}{\vt}}{\bracket{f}{\widecheck{\rho}^{\mathrm{pr}} \psi}}.
\label{NormMixedLebSchrod}
\end{equation}
It will be useful in the sequel to consider mixed Lebesgue norms in the following context:
	\begin{dfn} \label{MixLpq}
Let $G$ be a locally compact group given as the semi-direct product $G = N \rtimes H$ of two groups $N$ and $H$ (so that $N$ is isomorphic to some normal subgroup of $G$). Let $\delta: H \to \R^+$ be the continuous homomorphism for which we obtain the splitting $d\mu_G(n,h) = \delta^{-1}(h) \, d\mu_N(n) \, d\mu_H(h)$ (cf.~\cite[p.9]{kata12}). Given a weight $\vt$ on $G$ and $\pt, \qt  \in [1, \infty]$, we define the mixed-norm Lebesgue space
	\begin{align*}
		\LW{\pt, \qt}{G}{\vt} := \bigl\{ F: G \to \CF \mid F \mbox{ measurable and } \Norm{\LW{\pt, \qt}{G}{\vt}}{F} < \infty \bigr\}  
	\end{align*}
with
	\begin{align*}
		\Norm{\LW{\pt, \qt}{G}{\vt}}{F} := \Bigl(  \int_H \Bigl( \int_N \vt(n,h)^\pt \Abs{F(n,h)}^\pt \, d\mu_N(n) \Bigr)^{\qt/\pt} \delta^{-1}(h) \, d\mu_H(h) \Bigr)^{1/\qt}
	\end{align*}
if $\pt, \qt < \infty$ and the usual modifications otherwise. For $\LW{\pt, \pt}{G}{\vt}$ we also write $\LW{\pt}{G}{\vt}$ since the two spaces coincide.
	\end{dfn}

In Definition~\ref{MixLpq} we did not specify the type of weights $\vt$ that we are considering.
In the context of coorbit spaces, it is common to assume $\vt$ to be  $\wt$-moderate for a control weight $\wt$ which renders $\LW{\pt, \qt}{G}{\vt}$ a solid two-sided Banach convolution module over $\LW{1}{G}{\wt}$.

\begin{rem}\label{rem_dep_window_rholambda}
It is well known that the modulation spaces $\M{\pt, \qt}{\vt}{n}$ do not depend on the choice of window $\psi$ 
since the norms 
in \eqref{NormSplitExpSchrod} or \eqref{NormMixedLebSchrod} are equivalent
for different non-trivial windows (here Schwartz functions).
It is also true that instead of the Schr\"odinger representation $\rho$ 
we can choose any member of the family of Schr\"odinger representations $\rho_{\lambda}, \l \in \R \setminus \{ 0 \}$ (see Section \ref{subsec_family_Schrodinger}).
\end{rem}

\subsubsection{Modulation spaces on $\R^{n}$ defined as decomposition spaces}
We now have a closer look at \eqref{FourierInvSchrod} and \eqref{NormSplitExpSchrod},
and highlight fundamental connections between the representation-theoretic coorbit picture of modulation spaces and the Fourier-analytic decomposition space picture. 
In terms of notation,
we follow the conventions of \cite{fegr85, vo16-1}.
Let us emphasise the fact that, in order to take the decomposition space-theoretic viewpoint, it is crucial to replace $\rho^{pr}$ by $\widecheck{\rho}^{\mathrm{pr}}$.

We  assume that  the weight $\vt$ is a function of the frequency variable $p$ only:
$$
\vt(q,p)=\vt(p).
$$
The modulation space norm with window $\psi$ is equivalent to a decomposition space norm
	\begin{align}
		\Norm{\DS{Q}{}{\pt}{\qt}{\ut}}{f} = \Norm{\ell^\qt_\ut(I)}{\Bigl( \Norm{\L{\pt}{\R^n}}{\F^{-1} \bigl( \varphi_i \hspace{2pt} \widehat{f} \hspace{2pt} \bigr)}\Bigr)_{i \in I}} \label{DecSpNormModSp}
	\end{align}
	since we have for any  $f \in \SF{n}$ 
	\begin{align*}
		\subbracket{\L{2}{\R^d}}{f}{\widecheck{\rho}^{\mathrm{pr}} (q, p) \psi} = \subbracket{\L{2}{\widehat{\R}^n}}{\widehat{f}}{\rho^{\mathrm{pr}} (q, p) \widehat{\psi}} = \bigl( \F^{-1} \widehat{\overline{\psi}}(\hspace{2pt} . \hspace{2pt} + p) \F f \bigr)(-q),
	\end{align*}
allowing us to use well-known properties of  Fourier multipliers. 
The letter $\mathscr{Q}$ in \eqref{DecSpNormModSp} denotes a well-spread and relatively separated family of compacts $\{ Q_i \}_{i \in I}$ which cover the frequency space $\{ p \in \widehat{\R}^n \}$. Since the $Q_i$ can be chosen to be the translates of a fixed set $Q$ along a convenient lattice $\Lambda \subgr \widehat{\R}^n$, one can employ a partition of unity whose constituents $\varphi_i$ are compactly supported in the sets $Q_i$ in order to estimate the $p$-integral in \eqref{NormSplitExpSchrod} from above and below by a weighted $\ell^\qt$-sum of integrals over the compacts $Q_i$, for which the localised pieces of $\vt$ can be replaced by an equivalent discrete weight $\ut = \{ \ut(i) \}_{i \in I}$. 

\subsubsection{Besov spaces on $\R^{n}$}
A similar translation of the coorbit picture into an equivalent decomposition space picture can be applied to wavelet coorbit spaces, such as the homogeneous Besov spaces $\hB{\pt, \qt}{\st}{n}$ and the shearlet coorbit spaces. 

A generic, but fixed class of wavelet coorbit spaces is constructed in the following way.
Consider  the group $G := \R^n \rtimes H$,
where $H$ is a subgroup of $GL(\R^n)$, the so-called dilation group. 
Its representation $\tau$ 
acts on $\L{2}{\widehat{\R}^n}$ via
	\begin{align}
		\Bigl( \tau(x, h) \varphi \Bigr)(\xi) = \Abs{\det(h)}^{1/2} \hspace{2pt} e^{-2 \pi i x \xi} \hspace{2pt} \varphi \bigl( h \hspace{2pt} \xi \bigr). \label{IndQR}
	\end{align}
Let us define the intertwined representation 
$$
\widecheck{\tau} := \F \tau \F^{-1},
$$
which acts on $\L{2}{\R^n}$ via
	\begin{align*}
		\Bigl( \widecheck{\tau}(x, h) \psi \Bigr)(x) = \Abs{\det(h)}^{-1/2} \hspace{2pt} \psi(h^{-T}y - x).
	\end{align*}
The square-integrability and irreducibility of $\widecheck{\tau}$ is equivalent to $H$ satisfying certain admissibility conditions.
Given a weight $\vt$ on $H$, 
the membership of a distribution $f$ in the coorbit $\Co_{\widecheck{\tau}}(\LW{\pt, \qt}{G}{\vt})$ of $\LW{\pt, \qt}{G}{\vt}$ under $\widecheck{\tau}$ is determined via the norm
	\begin{align*}
		\Norm{\Co_{\widecheck{\tau}}
	(\LW{\pt, \qt}{G}{\vt})}{f} 
	&=
	\Norm{\LW{\pt, \qt}{G}{\vt}}{\bracket{f}{\widecheck{\tau} \ \psi}}
\\&=
	 \Bigl(  \int_{H} \Bigl( \int_{\R^n} \vt(h)^\pt \Abs{\bracket{f}{\widecheck{\tau} (x, h) \psi}}^\pt \, dx \Bigr)^{\qt/\pt} \, \frac{dh}{\Abs{\det(h)}} \Bigr)^{1/\qt}.
	\end{align*}
The transition to the equivalent decomposition space norm is performed via a covering of the dual orbit $\Orbit = H^T \xi_0 \subset \widehat{\R}^n$, which is an open set of full measure. Here $\bracket{\, . \,}{ \, . \,}$ denotes the extension of the $\L{2}{\R^n}$-inner product to the space of distributions $\F^{-1} \big( \mathcal{D}' (\Orbit) \bigr)$, where $\mathcal{D}(\Orbit) := C^\infty_c(\Orbit)$.

The homogeneous Besov spaces $\hB{\pt, \qt}{\st}{n}$, for example, are obtained by choosing $H = \R^+ \times SO(n)$ (cf.~Gr\"{o}chenig~\cite[\SS~3.3]{Gro91}). In this case, the isotropic action on $\R^n$ induces coverings which are equivalent to the isotropic dyadic coverings used in the usual decomposition space-theoretic description of $\hB{\pt, \qt}{\st}{n}$.

While homogeneous Besov spaces have been studied for decades, shearlet coorbit spaces have been introduced and studied only recently (see e.g., \cite{dakustte09, dastte12}). A general theory for wavelet coorbit spaces is due to the recent foundational paper F\"{u}hr~\cite{fu15}.
The equivalence of the two descriptions for homogeneous Besov spaces had in fact been known before the introduction of decomposition spaces, whereas the decomposition space-theoretic description in the much harder general case was established recently by F\"{u}hr and Voigtlaender~\cite{fuvo15}. Note that F\"{u}hr and Voigtlaender denote the quasi-regular representations $\widecheck{\tau}$ by $\pi$, whereas $\widecheck{\tau}$ was chosen to fit within our narrative.

\subsubsection{Guideline}
In the previous paragraphs, we have observed that 
the standard definitions of modulation spaces and wavelet coorbit spaces, in particular homogeneous Besov spaces, on $\R^n$ are in each case given as a family of spaces of distributions $f$ on $\R^n$ whose norms are defined via
	\begin{align*}
		\Norm{}{f} := \Norm{\LW{\pt, \qt}{G}{\vt}}{\bracket{f}{\widecheck{\sigma}\ \varphi}}.
	\end{align*}
More precisely, each family is defined via the $L^{\pt, \qt}$-norms of the matrix coefficients of a (possibly projective) square-integrable unitary irreducible group representations $\widecheck{\sigma}$ of a semi-direct product $G = N \rtimes H$ with $N = \R^n$ in which the exponent $\qt$ is assigned to the integral over $N$ whereas the exponent $\pt$ is assigned to the integral over $H$.
In each case the action of $\sigma := \F^{-1} \widecheck{\sigma} \hspace{1pt} \hspace{1pt} \F$ on $\L{2}{\widehat{\R}^n}$ gives access to a decomposition space-theoretic viewpoint.

\subsection{A First Definition} \label{Subsection_FirstDef}

In this subsection we give a first definition of Heisenberg-modulation spaces which is guided by principles we extract from the discussion in the previous subsection.

Accordingly, our first task is to write the Dynin-Folland group as a semi-direct product. For this, we write the element of the group in split coordinates:

	\begin{lem} \label{LemPrRepDF}
	The Dynin-Folland group may be realised as the semi-direct product
$$
\HG{n}{2} = \R^{2n+2} \rtimes \H,
$$
by  writing the elements of $\HG{n}{2}$ as products $(0, \mathcal{Q}, \mathcal{S}) \odot_{\HG{n}{2}} (\mathcal{P}, 0 ,0)$. We denote such an element by 
$$
\bigl( (\mathcal{Q}, \mathcal{S}), \mathcal{P} \bigr):=(0, \mathcal{Q}, \mathcal{S}) \odot_{\HG{n}{2}} (\mathcal{P}, 0 ,0).
$$
 Then the group law in these coordinates, to which we refer as split exponential coordinates for $\HG{n}{2}$, is given by
	\begin{align*}
		\bigl( (\mathcal{Q}, \mathcal{S}), \mathcal{P} \bigr) 
		\odot_{\HG{n}{2}} 
		\bigl( (\mathcal{Q}', \mathcal{S}'), \mathcal{P} '\bigr) 
		= \bigl( (\mathcal{Q} + \mathcal{Q}' + \frac{1}{2} {\coad}_{\H}(\mathcal{P})(\mathcal{Q}'), \mathcal{S} + \mathcal{S}' + \bracket{\mathcal{P}}{\mathcal{Q}'}), \mathcal{P} \HP \mathcal{P}' \bigr).
	\end{align*}
The generic representation $\pi_\l$, $\l \in \R\setminus \{ 0 \}$, realised in split exponential coordinates and acting on $\L{2}{\H}$ is given by
$$
		\Bigl( \pi_\l \bigl( (\mathcal{Q}, \mathcal{S}), \mathcal{P} \bigr) \psi \Bigr)(X) = e^{2 \pi i \l (\mathcal{S} + \bracket{\mathcal{Q}}{X})} \hspace{2pt} \psi(X \HP \mathcal{P}). 
		$$
		The corresponding projective representation is given by
	\begin{align*}
		\Bigl( \pi_\l^{\mathrm{pr}}(\mathcal{Q}, \mathcal{P}) \psi \Bigr)(X) = e^{2 \pi i \l \bracket{\mathcal{Q}}{X}} \hspace{2pt} \psi(X \HP \mathcal{P}).
	\end{align*}
Moreover, the Haar measure $\mu_{\HG{n}{2}}$ in split exponential coordinates is given by $d\mu_{\HG{n}{2}}\bigl( (\mathcal{Q}, \mathcal{S}), \mathcal{P} \bigr) = d\mathcal{S} \, d\mathcal{Q} \, d\mathcal{P}$.
	\end{lem}

	\begin{proof}
The group law in split exponential coordinates follows from direct  computations. 

As in Proposition~\ref{prop_pil_F}, we obtain the unirrep $\pi_\l$ by induction from the character $\rho_{F,\bfL}: \bigl( (\mathcal{Q}, \mathcal{S}), 0 \bigr) \mapsto e^{2 \pi i \l \mathcal{S}}$; the critical ingredient is the product
	\begin{align*}
		\bigl( (0, 0), X \bigr) \odot_{\HG{n}{2}} \bigl( (\mathcal{Q}, \mathcal{S}), \mathcal{P} \bigr) = \bigl( (\mathcal{Q} + \frac{1}{2} {\coad}_{\H}(X)(\mathcal{Q}), \mathcal{S} + \bracket{X}{\mathcal{Q}}), X \HP \mathcal{P} \bigr),
	\end{align*}
from which the expression for the representation $\pi_{\lambda}$ follows right away. The projective representation is obtained by quotienting the centre as in Corollary~\ref{CorProjReps}.

The Haar measure factors as $d\mu_{\R^{2n+2}}(\mathcal{Q}, \mathcal{S}) \, d\mu_{\R^{2n+2}}(\mathcal{P}) = d(\mathcal{Q}, \mathcal{S}) \, d\mathcal{P}$, the product of the Haar measures in exponential coordinates of $\R^{2n+2}$ and $\H$, respectively. This is owed to fact that because of the semi-direct product structure a cross-section for the quotient $\R^{2n+2} \backslash \HG{n}{2} \cong \H$ can be given in exponential coordinates (cf.~\cite[p.~23~Rem.~3]{CorwinGreenleaf}); finally, the central coordinate $\mathcal{S}$ can be split off again for the same reason.
	\end{proof}

We now consider the representation $\pi=\pi_{1}$ and its projective representation  $\pi^{\mathrm{pr}}
=\pi_1^{\mathrm{pr}}$, see Lemma~\ref{LemPrRepDF} above.
Intertwining 
with the $(2n+1)$-dimensional Euclidean Fourier transform $\F$, 
we define the projective representation
	\begin{align*}
		\widecheck{\pi}^{\mathrm{pr}} := \F^{-1} \hspace{2pt} \pi^{\mathrm{pr}} \hspace{2pt} \F.
	\end{align*}
	One checks readily that it  acts on $L^{2}(\R^{2n+1})$ via
	\begin{align*}
		\Bigl( \widecheck{\pi}^{\mathrm{pr}}(\mathcal{Q}, \mathcal{P}) \psi \Bigr)(X) = e^{2 \pi i \bracket{\mathcal{Q} - \mathcal{P}}{X}} \hspace{1pt} \psi \Bigl(X + \coad(\mathcal{P}) X \Bigr).
	\end{align*}
We recall that the Schwartz space $\SF{2n+1}$ coincides with the space of smooth vectors of $\widecheck{\pi}$, which we consider a reasonably good choice of test functions for our first definition.

	\begin{dfn} \label{1stDef}
Let $\pt, \qt \in [1, \infty]$ and let $\psi \in \SF{2n+1}$. We define the space $\E{\pt, \qt}{}{n}$ by
	\begin{align*}
		\E{\pt, \qt}{}{n} := \bigl \{ f \in \SD{2n+1} \mid \Norm{\E{\pt, \qt}{}{n}}{f} < \infty \bigr\}
	\end{align*}
with
$$
\Norm{\E{\pt, \qt}{}{n}}{f} :=
\Norm{\LW{\pt, \qt}{\HG{n}{2}/Z(\HG{n}{2})}{\vt}}{\bracket{f}{\widecheck{\pi}^{\mathrm{pr}} \psi}}.
$$
We refer to the spaces $\E{\pt, \qt}{}{n}$ as \emph{Heisenberg-modulation spaces or $\H$-modulation spaces}.
	\end{dfn}

The name of these spaces is motivated by the Heisenberg-type of frequency shifts $\pi^{\mathrm{pr}}(0, \mathcal{P}), \mathcal{P} \in \H,$ we indirectly make use of.
The  letter `E' for the notation is generic and stands for `espace' or `espacio'.

More explicitly, the $\H$-modulation space norm is given by:
	\begin{align*}
		\Norm{\E{\pt, \qt}{}{n}}{f} :=& 
		\left( \int_{\H} \left( \int_{\R^{2n+1}} \Abs{\subbracket{\SDG{\H}}{f}{\widecheck{\pi}^{\mathrm{pr}}(\mathcal{Q}, \mathcal{P}) \psi}}^\pt \,d\mathcal{Q} \right)^{\qt/\pt} d\mathcal{P} \right)^{1/\qt} \\
		=& \left( \int_{\H} \left( \int_{\R^{2n+1}} \Abs{\int_{\R^{2n+1}} f(X) \hspace{1pt} e^{-2 \pi i \bracket{ \mathcal{Q} - \mathcal{P}}{X}} \hspace{1pt} \overline{\psi} \Bigl(X + \coad(\mathcal{P}) X \Bigr) \,dX}^\pt \,d\mathcal{Q} \right)^{\qt/\pt} d\mathcal{P} \right)^{1/\qt}
	\end{align*}
if $\pt, \qt < \infty$ and the usual modifications otherwise. 

\medskip

In the last two sections of this paper we develop a decomposition space-theoretic picture and a coorbit-theoretic picture of the spaces $\E{\pt, \qt}{}{n}$. Let us mention two aspects of the following results.
Firstly, this will imply (see Section~\ref{CS_DF})  that, as in $\R^{n}$ (see Remark \ref{rem_dep_window_rholambda}), the specific choice of $\pi_\l = \pi_1 = \pi$  for $\l = 1$  is unnecessary, and that distinct non-vanishing windows $\psi$ will give equivalent norms.
Secondly, by building upon recent deep results in decomposition space theory by Voigtlaender~\cite{vo16-1}, our descriptions will lead to a conclusive comparison of the spaces $\E{\pt, \qt}{}{n}$ with modulation spaces and  Besov spaces on $\R^{2n+1}$ in the final subsection.

\section{The Decomposition Spaces $\EE{\pt, \qt}{\st}{n}$} \label{DS_E}

In this section we introduce a new class of decomposition spaces. 
 For the sake of brevity, we abstain from recalling the essentials of decomposition space theory and instead refer to Feichtinger and Gr\"{o}bner's foundational paper \cite{fegr85} and Voigtlaender's recent opus magnum \cite{vo16-1}. As their notations and conventions differ very little, we will adhere to the more recent account \cite{vo16-1}.

\medskip


In analogy to the frequently used weights
	\begin{align}
		\vt_\st(p) := (1 + \Abs{p}^2)^{\st/2}, \hspace{5pt} \st \in \R, \label{v_s}
	\end{align}
in modulation space theory, we define weights on $\H$ whose polynomial growth is controlled by powers of the homogeneous Cygan-Koranyi quasi-norm
	\begin{align*}
		| \hspace{7pt} |_{\H}: \H &\to [0, \infty), \\
		(p, q, t) &\mapsto \Bigl( (\Abs{p}^2 + \Abs{q}^2)^2 + 16 \hspace{1pt} t^2 \Bigr)^{1/4}.
	\end{align*}
Like every quasi-norm it is continuous, homogeneous of degree $1$ with respect to the natural dilations on $\H$, symmetric with respect to $X \mapsto X^{-1}$ and definite, i.e. $\Abs{X} = 0$ if and only if $\Abs{X} = 0$. The fact that this quasi-norm is sub-additive, thus in fact a norm, was first proved in Cygan~\cite{Cyg83}. For more details on quasi-norms we refer to the monograph Fischer and Ruzhansky~\cite[\SS~3.1.6]{FiRuz}. 

	\begin{dfn} \label{PolyHeisWeights}
For $\st \in \R$ we define the weight $\hvt_\st$ by
	\begin{align*}
		&\hvt_\st: \H \to (0, \infty), \\
		&\hvt_\st(\mathcal{P}) := (1 + \Abs{\mathcal{P}}^4_{\H})^{\st/4}.
	\end{align*}
	\end{dfn}
Since any two quasi-norms on $\H$ are equivalent, different choices of quasi-norm in Definition~\ref{PolyHeisWeights} yield equivalent weights $\hvt_\st$.

To arrive at an admissible covering of $\Orbit = \widehat{\R}^{2n+1}$ and a convenient $\L{1}{\R^{2n+1}}$-bounded admissible partition of unity (BAPU), we make use of a discrete subgroup $\Gamma \subgr \H$. 

We define 
\begin{equation}
\Gamma := \exp_{\H}\Bigl( \{ (a, b, c) \in \widehat{\R}^{2n+1} \mid a, b \in (2 \Z)^n, c \in 2 \Z \} \Bigr).
\label{Gamma}
\end{equation}
It is well known (see e.g. \cite[SS 5.4]{CorwinGreenleaf})
that $\Gamma$ is a lattice subgroup of $\H$
in the sense that it is discrete and cocompact in $\H$ and that 
 $\{ (a, b, c) \in \Z^{2n+1} \mid a, b \in (2 \Z)^n, c \in 2 \Z \}$ is an additive subgroup of $\h$.
Furthermore, its fundamental domain is  
$$
\Sigma := \exp_{\H}\Bigl( [0, 2)^{2n+1} \Bigr).
$$

Let us construct an admissible covering with subordinate BAPU around this lattice:

	\begin{prop} \label{PropDSDF}
Let $\Gamma$ be the lattice as above in \eqref{Gamma}.
For a chosen $\e \in (0, \frac{1}{2})$,
we fix   a non-negative smooth function
$\tilde{\vartheta}$ 
supported inside $P := (-\e, 2 + \e)^{2n+1} \subseteq \widehat{\R}^{2n+1}$ which equals $1$ on $(-\frac{\e}{2}, 2 + \frac{\e}{2})^{2n+1}$ and set 
$$
\tilde{\vartheta}_\gamma(\Xi) := \tilde{\vartheta}(\Xi \HP \gamma^{-1})
\qquad\mbox{as well as}\qquad 
\vartheta_\gamma := \tilde{\vartheta}_\gamma/(\sum_{\gamma' \in \Gamma} \tilde{\vartheta}_{\gamma'}).
$$ 
Finally, for $\st \in \R$ we define $\hut_\st := \{ \hut_\st(\gamma) \}_{\gamma \in \Gamma}$ by $\hut_\st(\gamma) := \hvt_\st(\gamma)$.
Then
	\begin{itemize}
		\item[(i)] $\mathscr{P} := \{ P \HP \gamma\}_{\gamma \in \Z^{2n+1}}$ is a structured covering of $\Orbit := \widehat{\R}^{2n+1}$.
		\item[(ii)] $\Theta := \{ \vartheta_\gamma \}_{\gamma \in \Gamma}$ is an $L^1$-BAPU subordinate to $\mathscr{P}$ with $\supp (\vartheta_\gamma) \subset P \HP \gamma$ for all $\gamma \in \Gamma$.
		\item[(iii)] The weight $\hut_\st$ is $\mathscr{P}$-moderate and we have 
		 for all $\gamma \in \Gamma$ and $\Xi \in P_\gamma$:
$$
\hut_\st(\gamma) \asymp \hvt_\st(\Xi) \asymp \int_{P_\gamma} \hvt_\st(\Xi') \, d\Xi'.
$$
	\end{itemize}
	\end{prop}

In the proof of Proposition \ref{PropDSDF}
and other arguments, 
we will use the following observation.
\begin{lem}
\label{lem_Hmult}
The $\H$-group multiplication can be expressed by affine linear maps:
	\begin{align}
		\Xi \HP \gamma
	=
		\begin{pmatrix} 
				\xi + a \\
				\omega + b \\
				\tau + c +\frac{1}{2}(\xi b - \omega a)
		\end{pmatrix}
	=
		\begin{pmatrix} 
				1 & 0 & 0 \\
				0 & 1 & 0 \\
				\frac{1}{2} b & -\frac{1}{2} a & 1
		\end{pmatrix}
		\begin{pmatrix} 
				\xi \\
				\omega \\
				\tau
		\end{pmatrix}
	+
		\begin{pmatrix} 
				a \\
				b \\
				c
		\end{pmatrix}
	=: T_\gamma \hspace{2pt} \Xi + \gamma \label{T_gamma}
	\end{align}
		for any $\Xi = (\xi, \omega, \tau)$ and $\gamma = (a, b, c)$ in $\H$. 
		The $(2n+1)\times (2n+1)$ matrix $T_{\gamma}$ has determinant 1.
		We also have
 $$
 T_\gamma^{-1} = T_{\gamma^{-1}} = T_{-\gamma}, 
 \quad\mbox{and}\quad 
T_{\gamma}T_{\gamma'}=T_{\gamma+\gamma'}=T_{\gamma \HP\gamma'}.
$$ 
\end{lem}
Lemma \ref{lem_Hmult} follows from easy computations left to the reader.

	\begin{proof}[Proof of Proposition \ref{PropDSDF} Part (i)]
 Since $\Gamma$ is a lattice subgroup with 
$$
\overline{\Sigma \HP \gamma} \subset P_\gamma := P \HP \gamma,
$$
 the family of non-empty, relatively compact, connected, open sets $\mathscr{P}$ constitutes a covering of $\Orbit$. 
  If $P_\gamma \cap P_{\gamma'} \neq \emptyset$, then there exist $\Xi, \Xi' \in P$ such that $\Xi \HP \gamma = \Xi' \HP \gamma'$.
The equality $\Abs{\Xi'^{-1} \HP \Xi}_{\H} = \Abs{\gamma' \HP \gamma^{-1}}_{\H}$ yields the estimate
	\begin{align*}
		\Abs{\gamma' \HP \gamma^{-1}}_{\H} \leq \max_{\Xi, \Xi' \in P} \Abs{\Xi'^{-1} \HP \Xi}_{\H} \leq 2 \max_{\Xi \in P} \Abs{\Xi}_{\H} \leq 2 \hspace{1pt} \bigl( 4 \hspace{1pt} (2 + \e)^4 + 16 \hspace{1pt}  (2 + \e)^2 \bigr)^{1/4} \leq 9.
	\end{align*}
	So for a fixed $\gamma$ there are only finitely many $\gamma' \in \Gamma$ with $\Abs{\gamma' \HP \gamma^{-1}}_{\H} \leq 9$ and the argument is independent of $\gamma$.
	Therefore, the following supremum is finite:
	 $$
 N_\mathscr{P} := \sup_{\gamma \in \Gamma} \Abs{\gamma^*} < \infty
 \qquad\mbox{where}\qquad 
 \gamma^* := \{ \gamma' \in \Gamma \mid P_\gamma \cap P_{\gamma'} \neq \emptyset \},
 $$
and the covering is admissible.

The same argument gives an admissible covering for the open set $(-\frac{\e}{2}, 2 + \frac{\e}{2})^{2n+1} \subseteq P$. 

Lemma \ref{lem_Hmult} easily implies that
$\sup_{\gamma \in \Gamma} \sup_{\gamma' \in \gamma^*} \Norm{}{T_\gamma^{-1} T_{\gamma'}} < \infty$, 
and this shows that $\mathscr{P}$ is a structured covering of $\Orbit = \widehat{\R}^{2n+1}$.
\end{proof}

	\begin{proof}[Proof of Proposition \ref{PropDSDF} Part (ii)]
 Since the conditions~$(i)$ - $(iii)$ from \cite[Def.~3.5]{vo16-1} follow from the construction of $\Theta$, we only have to show the condition~$(vi)$, that is, the uniform bound on the Fourier-Lebesgue norm $\Norm{\L{1}{\R^{2n+1}}}{\F^{-1} \vartheta_\gamma}$. 
 The latter follows from $
\Norm{\L{1}{\R^{2n+1}}}{\F^{-1} \vartheta_\gamma} = \Norm{\L{1}{\R^{2n+1}}}{\F^{-1} \vartheta_0}$, which holds since the group multiplication is given by an affine 
transformation with determinant 1, see  Lemma \ref{lem_Hmult}.
\end{proof}

	\begin{proof}[Proof of Proposition \ref{PropDSDF} Part (iii)]
	Part (iii) follows easily from 
 the triangle inequality for the Cygan-Koranyi norm and the compactness of $P$.
	\end{proof}

Our new class of decomposition spaces is defined by the above ingredients.

	\begin{dfn} \label{DefEspaces}
Let $\mathscr{P}$ be the admissible covering of $\widehat{\R}^{2n+1}$ with BAPU $\Theta$ defined by Proposition~\ref{PropDSDF}. Then for $\pt, \qt \in [1, \infty]$ and $\st \in \R$ we denote the decomposition space $\DS{P}{}{\pt}{\qt}{\hut_\st}$ by $\EE{\pt, \qt}{\st}{n}$. For $\EE{\pt, \pt}{\st}{n}$ we may also write $\EE{\pt}{\st}{n}$.
	\end{dfn}

\begin{figure}[hp]
	\begin{center}
		\includegraphics[width=.35\textwidth]{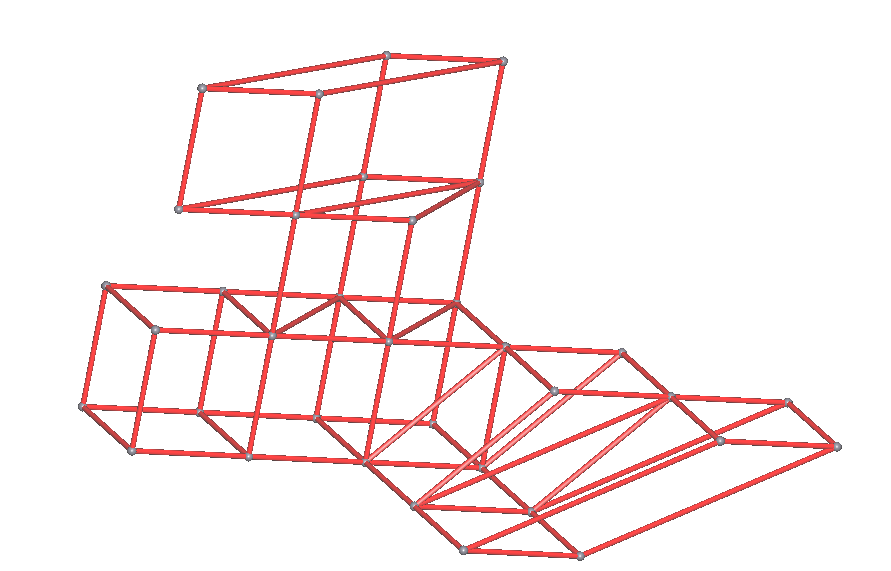}
		\caption{Detail of the Lattice Subgroup $\Gamma$ of $\mathbf{H}_1 \cong \R^3$.}
	\end{center}
\end{figure}

\section{The Coorbit Spaces $\E{\pt, \qt}{\st}{n}$} \label{CS_E}

In this final section we study the coorbit spaces derived from the projective representations of the Dynin-Folland group $\HG{n}{2}$. 
In the first subsection we characterise the coorbits $\E{\pt, \qt}{\st}{n}$ derived from the generic representation and observe that the $\H$-modulation spaces coincide with the unweighted coorbits $\E{\pt, \qt}{0}{n}$. In the second subsection we give a complete classification and characterisation of the coorbits related to $\HG{n}{2}$. In the third subsection we provide a conclusive comparison of the coorbits $\E{\pt, \qt}{\st}{n}$ with the classical modulation spaces and Besov spaces on $\R^{2n+1}$. The comparison is based on Voigtlaender's decomposition space paper \cite{vo16-1}.

Proceeding as in Section~\ref{DS_E}, we abstain from recalling the basics of coorbit theory and instead refer to Feichtinger and Gr\"{o}chenig's foundational paper~\cite{fegr89} as well as Gr\"{o}chenig's paper~\cite{Gro91}.

\subsection{The Generic Coorbits Related to $\HG{n}{2}$} \label{GenCoorbits}

In this subsection we define the generic coorbit spaces under the projective representation $\widecheck{\pi}^{\mathrm{pr}}$ of the Dynin-Folland group $\HG{n}{2}$ and state a decomposition space-theoretic characterisation.

	\begin{lem} \label{LemHWeights}
For the weights $\hvt_\st, \st \in \R$, on $\H$ from Definition~\ref{PolyHeisWeights} we have the following:
	\begin{itemize}
		\item[(i)] If $\st \in [0, \infty)$, then $\hvt_\st$ is submultplicative.
		\item[(ii)] If $\st \in (-\infty, 0)$, then $\hvt_\st$ is $\hvt_{-\st}$-moderate.
		\item[(iii)] The extensions of $\hvt_\st$ to $\HG{n}{2}$ and $\HG{n}{2}/Z(\HG{n}{2})$, which we also denote by $\hvt_\st$, satisfy (i) and (ii) for the respective (quotient) group structures.
	\end{itemize}
Let $\hvt_{\st_1}$ be $\hvt_{\st_2}$-moderate for $\st_1, \st_2 \in \R$. Then for every $\pt, \qt \in \R$ the space $\LW{\pt, \qt}{\HG{n}{2}/Z(\HG{n}{2})}{\hvt_{\st_1}}$ is an
 $\LW{1}{\HG{n}{2}/Z(\HG{n}{2})}{\hvt_{\st_2}}$-Banach convolution module.
	\end{lem}

	\begin{proof}
Since for the $\HG{n}{2}$-Haar measure in split-exponential coordinates $\delta = 1$, the proofs are identical to the proofs of the standard Euclidean statements in \cite{gr01}, Lemmas~11.1 and 11.2.
	\end{proof}

	\begin{dfn} [The Generic Coorbits] \label{GenCo}
Let $\st \in \R$, $\pt, \qt \in [1, \infty]$ and let $\widecheck{\pi}^{\mathrm{pr}}$ be the projective representation of the Dynin-Folland group $\HG{n}{2}$ from Subsection~\ref{Subsection_FirstDef}. We define the generic coorbit space related to $\HG{n}{2}$ by
	\begin{align*}
		\E{\pt, \qt}{\st}{n} := \Co_{\widecheck{\pi}^{\mathrm{pr}}} \Bigl( \LW{\pt, \qt}{\HG{n}{2}/Z(\HG{n}{2})}{\hvt_\st} \Bigr) := \Co_{\widecheck{\pi}^{\mathrm{pr}}} \Bigl( \LW{\pt, \qt}{\R^d \rtimes \H/\exp(\R X_s)}{\hvt_\st} \Bigr).
	\end{align*}
	\end{dfn}

We can announce our first main theorem.

	\begin{thm} \label{MainThm1}
	Let $\pt, \qt \in [1, \infty]$. 
The case $\st=0$ of the coorbit spaces $\E{\pt, \qt}{\st}{n}$ from Definition~\ref{GenCo} coincides with the space $\E{\pt, \qt}{}{n}$ from Definition~\ref{1stDef}:
$$
\E{\pt, \qt}{0}{n} = \E{\pt, \qt}{}{n}.
$$ 
For any  $\st \in \R$,  
the coorbit space $\E{\pt, \qt}{\st}{n}$
coincides with the 
 decomposition space  $\EE{\pt, \qt}{\st}{n}$  from Definition~\ref{DefEspaces}:
 $$
\E{\pt, \qt}{\st}{n} = \EE{\pt, \qt}{\st}{n}.
$$
Moreover, Definition~\ref{GenCo} and, a fortiori, Definition~\ref{1stDef}, do not depend on the specific choices of representation $\widecheck{\pi}^{\mathrm{pr}}_\l, \l \in \R \setminus \{ 0 \}$, window $\psi \in \SF{2n+1} \setminus \{ 0 \}$, and quasi-norm $\Abs{\, . \,}_{\H}$.
	\end{thm}

The proof will be a consequence of Proposition~\ref{MainThm2}~Case (1). An essential step in the proof is to show that for every $\st \in \R$ the space $\SF{2n+1}$ can be taken as the space of windows $\psi$ and its topological dual $\SDG{\H} \cong \SD{2n+1}$ is a sufficiently large reservoir of distributions for the coorbit-theoretic definition of $\E{\pt, \qt}{\st}{n}$.

\subsection{A Classification of the Coorbits Related to $\HG{n}{2}$} \label{CS_DF}

We recall that by \cite[Thm.~4.8~(i)]{fegr89} two coorbits are isometrically isomorphic if the defining representations are unitarily equivalent. This is important because Proposition~\ref{MainThm2}, which classifies the coorbit spaces related to the Dynin-Folland group according to the projective unirreps of $\HG{n}{2}$, also provides decomposition space-theoretic characterisations of the coorbits. As a consquence, these characterisations permit a conclusive comparison of the spaces with classical modulation spaces and Besov spaces via the decomposition space-theoretic machinery developed in \cite{vo16-1}.

	\begin{prop} \label{MainThm2}
Let $\pt, \qt \in [1, \infty]$ and $\st \in \R$. Let $\vt_\st$ and $\hvt_\st$ be defined as in \eqref{v_s} and Definition~\ref{PolyHeisWeights}, respectively. Let $\pi_\l^{\mathrm{pr}}$ be the projective representation of $\HG{n}{2}$ from Lemma~\ref{LemPrRepDF} and set $\widecheck{\pi}_\l^{\mathrm{pr}} := \F^{-1} \hspace{2pt} \pi_\l^{\mathrm{pr}} \hspace{2pt} \F$ for the $(2n+1)$-dimensional Euclidean Fourier transform $\F$. Moreover, let $\EE{\pt, \qt}{\st}{n}$ and $\E{\pt, \qt}{\st}{n}$ be defined as in Definition~\ref{DefEspaces} and Definition~\ref{GenCo}.
Then, up to isometric isomorphisms, we have the following classification and characterisations of coorbits of $L^{\pt, \qt}$ under the projective unitary irreducible representations of the Dynin-Folland group $\HG{n}{2}$:
	\begin{itemize}
		\item[Case (1)] For $\pi_\l^{\mathrm{pr}}$ we have
	\begin{align*}
		\Co_{\widecheck{\pi}_\l^{\mathrm{pr}}} \Bigl( \LW{\pt, \qt}{\HG{n}{2}/Z(\HG{n}{2})}{\hvt_\st} \Bigr) = \E{\pt, \qt}{\st}{n} = \EE{\pt, \qt}{\st}{n}.
	\end{align*}
		\item[Case (2)] For $\pi_{{(f_w,f_x,f_y,f_z)}}^{\mathrm{pr}}$ we have
	\begin{align*}
		\Co_{\widecheck{\pi}^{\mathrm{pr}}_{(f_w,f_x,f_y,f_z)}} \Bigl( \LW{\pt, \qt}{\H/Z(\H)}{\vt_\st} \Bigr) = \Co_{\widecheck{\rho}_{f_w}^{\mathrm{pr}}} \Bigl( \LW{\pt, \qt}{\H/Z(\H)}{\vt_\st} \Bigr) = \M{\pt, \qt}{\vt_\st}{n}.
	\end{align*}
		\item[Case (3)]  For $\pi_{{(f_u,f_v,f_x,f_y)}}^{\mathrm{pr}}$ we have
			\begin{align*}
				\Co_{\widecheck{\pi}_{(f_u,f_v,f_x,f_y)}^{\mathrm{pr}}} \Bigl( \LW{\pt, \qt}{\mathbf{H}_1/Z(\mathbf{H}_1)}{\vt_\st} \Bigr) = \Co_{\widecheck{\rho}_{\l_{f_x, f_y}}^{\mathrm{pr}}} \Bigl( \LW{\pt, \qt}{\mathbf{H}_1/Z(\mathbf{H}_1)}{\vt_\st} \Bigr) = \M{\pt, \qt}{\vt_\st}{}.
			\end{align*}
	\end{itemize}
	\end{prop}

	\begin{proof}
Case $(1)$ To prove the assertion, we only need to prove $\Co_{\widecheck{\pi}_\l^{\mathrm{pr}}} \Bigl( \LW{\pt, \qt}{\HG{n}{2}/Z(\HG{n}{2})}{\hvt_\st} \Bigr) = \EE{\pt, \qt}{\st}{n}$ for an arbitrary but fixed $\l \in \R \setminus \{ 0 \}$.

To prove the equality, we show the equivalence of the norms of $\Co_{\widecheck{\pi}_\l^{\mathrm{pr}}} \Bigl( \LW{\pt, \qt}{\HG{n}{2}/Z(\HG{n}{2})}{\hvt_\st} \Bigr)$ and $\EE{\pt, \qt}{\st}{n}$. Our line of arguments will be very close to the argument which proves the equivalence of the coorbit space and decomposition spaces norms of the modulation spaces $\M{\pt, \qt}{\vt_\st}{2n+1}$. Since the latter is very well known, we keep our reasoning brief.
To start with, we first show that for a weight $\wt$ which renders $\hvt_\st$ $\wt$-moderate the space $\SF{2n+1}$ is a subset of the space of analysing vectors $\mathcal{A}_{\widecheck{\pi}_\l^{\mathrm{pr}}, \wt}$, the standard space of windows in coorbit theory, and of $\Co_{\widecheck{\pi}_\l^{\mathrm{pr}}} \Bigl( \LW{1}{\HG{n}{2}/Z(\HG{n}{2})}{\wt} \Bigr)$, the standard space of test functions in coorbit theory. By combining Lemma~\ref{LemHWeights}~(i) and (ii), $\hvt_\st$ is $\hvt_{\Abs{\st}}$-moderate, which makes $\hvt_{\Abs{\st}}$ a natural choice of control weight $\wt$.
Moreover, for the voice transform $V_\psi(f) = \bracket{\psi}{\widecheck{\pi}_\l^{\mathrm{pr}}(\mathcal{Q}, \mathcal{P}) \psi}$ the formal computation
	\begin{align*}
		V_{\widehat{\psi_2}}(\widehat{\psi_1}) = \subbracket{\L{2}{\widehat{\R}^{2n+1}}}{\widehat{\psi_1}}{\pi_\l^{\mathrm{pr}}(\mathcal{Q}, \mathcal{P}) \widehat{\psi_2}} = \int_{\widehat{\R}^{2n+1}} \widehat{\psi_1}(\Xi) \hspace{2pt} e^{2 \pi i \l \bracket{\mathcal{Q}}{X}} \hspace{2pt} \widehat{\psi_2}(\Xi \HP \mathcal{P}) \,d\Xi
	\end{align*}
implies that the map $V: \L{2}{\R^{2n+1}} \otimes \L{2}{\R^{2n+1}} \to \L{2}{\R^{4n+2}}: (\widehat{\psi_1}, \widehat{\psi_2}) \mapsto V_{\widehat{\psi_2}}(\widehat{\psi_1})$ extends to a map $\tilde{V}$ defined on $\L{2}{\R^{4n+2}} \cong \L{2}{\R^{2n+1}} \otimes \L{2}{\R^{2n+1}}$. Given as the composition of the unitary coordinate transform $(\mathcal{T}F)(\mathcal{P}, \Xi) := F(\mathcal{P}, \Xi \HP \mathcal{P})$, $F \in \L{2}{\R^{4n+2}}$, 
and the partial $(2n+1)$-dimensional Euclidean Fourier transform in the second variable, $\tilde{V}$ itself is unitary. Clearly, it maps $\SF{4n+2}$ onto itself, hence $V$ maps $\SF{2n+1} \times \SF{2n+1}$ into $\SF{4n+2}$. So, for $\psi \in \mathscr{S}(\R^{2n+1})$ the function
	\begin{align*}
		V_{\widehat{\psi}}(\widehat{\psi}) = \subbracket{\L{2}{\widehat{\R}^{2n+1}}}{\widehat{\psi}}{\pi_\l^{\mathrm{pr}}(\mathcal{Q}, \mathcal{P}) \widehat{\psi}} = \subbracket{\L{2}{\R^{2n+1}}}{\psi}{\widecheck{\pi}_\l^{\mathrm{pr}}(\mathcal{Q}, \mathcal{P}) \psi}
	\end{align*}
has finite $\LW{\pt, \qt}{\HG{n}{2}/Z(\HG{n}{2})}{\hvt_\st}$-norm for every $\st \in \R$ and $\pt, \qt \in [1, \infty]$, i.e. \newline
$\SF{2n+1} \subseteq \Co_{\widecheck{\pi}_\l^{\mathrm{pr}}} \Bigl( \LW{\pt, \qt}{\HG{n}{2}/Z(\HG{n}{2})}{\hvt_\st} \Bigr)$. In particular, $\SF{2n+1}$ is a subset of $\mathcal{A}_{\widecheck{\pi}_\l^{\mathrm{pr}}, \wt}$ and $\Co_{\widecheck{\pi}_\l^{\mathrm{pr}}} \Bigl( \LW{1}{\HG{n}{2}/Z(\HG{n}{2})}{\wt} \Bigr)$. Consequently, $\SD{2n+1}$ is a superset of $\Co_{\widecheck{\pi}_\l^{\mathrm{pr}}} \Bigl( \LW{\infty}{\HG{n}{2}/Z(\HG{n}{2})}{\wt} \Bigr) = \Co_{\widecheck{\pi}_\l^{\mathrm{pr}}} \Bigl( \LW{\infty}{\HG{n}{2}/Z(\HG{n}{2})}{\hvt_{-\Abs{\st}}} \Bigr)$, the corresponding natural space of distributions, and we obtain the equivalent identification
	\begin{align*}
		\Co_{\widecheck{\pi}_\l^{\mathrm{pr}}} \Bigl( \LW{\pt, \qt}{\HG{n}{2}/Z(\HG{n}{2})}{\hvt_\st} \Bigr) = \Bigl \{ f \in \SD{2n+1} \mid V_\psi(f) \in \LW{\pt, \qt}{\HG{n}{2}/Z(\HG{n}{2})}{\hvt_\st} \Bigr\}
	\end{align*}
for an arbitrary but fixed window $\psi \in \SF{2n+1}$. Hence Definition~\ref{1stDef} is included in Definition~\ref{GenCo}.

As a consequence, we may compute the coorbit space norm with respect to a window $\psi \in \mathscr{S}(\R^{2n+1})$ such that its Fourier transform $\widehat{\psi} := \vartheta$ is a non-negative function supported inside $P = (-\e, 2 + \e)^{2n+1}$ which equals $1$ on $(-\frac{3 \e}{4}, 2 + \frac{3 \e}{4})^{2n+1}$. Hence, $\vartheta = \overline{\vartheta} = \overline{\widehat{\psi}}$, and for $f \in \SF{2n+1}$ we compute
	\begin{align}
			\Norm{\Co_{\widecheck{\pi}_\l^{\mathrm{pr}}} ( L^{\pt, \qt}_{\hvt_\st} )}{f}^{\qt}& = \int_{\R^{2n+1}} \Bigl( \int_{\R^{2n+1}} \hvt_\st(\mathcal{P})^\pt \Abs{\int_{\widehat{\R}^{2n+1}} \widehat{f}(\Xi) \hspace{2pt} e^{-2 \pi i \l \bracket{\mathcal{Q}}{\Xi}} \hspace{2pt} \overline{\widehat{\psi}}(\Xi \HP \mathcal{P}) \, d\Xi}^\pt \, d\mathcal{Q} \Bigr)^{\qt/\pt} \, d\mathcal{P} \nn \\
			&= \sum_{\gamma \in \Gamma} \int_{\Sigma \HP \gamma} \hvt_\st(\mathcal{P})^\qt \Bigl( \int_{\R^{2n+1}} \Abs{\int_{\widehat{\R}^{2n+1}} \widehat{f}(\Xi) \hspace{2pt} e^{-2 \pi i \l \bracket{\mathcal{Q}}{\Xi}} \hspace{2pt} \vartheta(\Xi \HP \mathcal{P}) \, d\Xi}^\pt \, d\mathcal{Q} \Bigr)^{\qt/\pt} \, d\mathcal{P}. \nn 
	\end{align}
To estimate $\Norm{\Co_{\widecheck{\pi}_\l^{\mathrm{pr}}} ( L^{\pt, \qt}_{\hvt_\st} )}{f}^{\qt}$ from above, we apply to
	\begin{align}
		\int_{\widehat{\R}^{2n+1}} \widehat{f}(\Xi) \hspace{2pt} e^{-2 \pi i \l \bracket{\mathcal{Q}}{\Xi}} \hspace{2pt} \vartheta(\Xi \HP \mathcal{P}) \, d\Xi = \Bigl( \F^{-1} \bigl( \widehat{f} \hspace{1pt} \pi_\l^{\mathrm{pr}}(0, \mathcal{P}) \vartheta \bigr) \Bigr) (-\l \mathcal{Q}) \label{IntRep}
	\end{align}
the change of variables $\mathcal{Q} \mapsto -\l \mathcal{Q}$, which produces a factor $\Abs{\l}^{-\frac{2n+1}{\pt}}$ in the coorbit norm. For general $f \in \SD{2n+1}$ we have to extend the integral representation \eqref{IntRep} to
	\begin{align*}
		\subbracket{\mathscr{S}'(\widehat{\R}^{2n+1})}{\widehat{f}}{\pi_\l^{\mathrm{pr}}(\mathcal{Q}, \mathcal{P}) \vartheta} = \Bigl( \F^{-1} \bigl( \widehat{f} \hspace{1pt} \pi_\l^{\mathrm{pr}}(0, \mathcal{P}) \vartheta \bigr) \Bigr) (-\l \mathcal{Q})
	\end{align*}
by duality. Now, an application of Young's convolution inequality yields the crucial estimate
	\begin{align}
		\Norm{\L{\pt}{\R^{2n+1}}}{\F^{-1} \bigl( \widehat{f} \hspace{1pt} \pi_\l^{\mathrm{pr}}(0, \mathcal{P}) \vartheta \bigr)} \leq \sum_{\gamma' \in \gamma^{**}} \Norm{\L{1}{\R^{2n+1}}}{\F^{-1} \bigl( \pi_\l^{\mathrm{pr}}(0, \mathcal{P}) \vartheta \bigr)} \Norm{\L{\pt}{\R^{2n+1}}}{\F^{-1} \bigl( \vartheta_\gamma \widehat{f} \hspace{2pt} \bigr)}. \label{EssEstAbove}
	\end{align}
The first factor is bounded by
	\begin{align}
		\Norm{\L{1}{\R^{2n+1}}}{\F^{-1} \bigl( \pi_\l^{\mathrm{pr}}(0, \mathcal{P}) \vartheta \bigr)} = \Norm{\L{1}{\R^{2n+1}}}{\F^{-1} \vartheta} < \infty \label{FL1NormTheta}
	\end{align}
for the following reason: to
	\begin{align*}
		\Bigl( \F^{-1} \bigl( \pi_\l^{\mathrm{pr}}(0, \mathcal{P}) \vartheta \bigr) \Bigr)(X) = \int_{\widehat{\R}^{2n+1}} e^{2 \pi i \bracket{ X}{\Xi}} \hspace{2pt} \vartheta(\Xi \HP \mathcal{P}) \,d\Xi
	\end{align*}
with $X = (p, q, t)$ and $\mathcal{P} = (u, v, w)$ we apply the change of variables $\Xi \mapsto \tilde{\Xi} := \Xi \HP \mathcal{P}$. By rewriting
	\begin{align*}
		\bracket{X}{\tilde{\Xi} \HP \mathcal{P}^{-1}} &= \bracket{X}{\tilde{\Xi} - \mathcal{P} - \frac{1}{2} [\tilde{\Xi}, \mathcal{P}]} \\
		&= -\bracket{X}{\mathcal{P}} + \bracket{X - \coad(\mathcal{P})(X)}{\tilde{\Xi}} \\
		&= -\bracket{(p, q, t)^T}{(u, v, w)^T} + \bracket{(p - \frac{1}{2} t v, \l q + \frac{1}{2} t u, t)^T}{\tilde{\Xi}},
	\end{align*}
we obtain
	\begin{align*}
		\Bigl( \F^{-1} \bigl( \pi_\l^{\mathrm{pr}}(0, \mathcal{P}) \vartheta \bigr) \Bigr)(p, q, t) = e^{-2 \pi i \l \bracket{(p, q, t)^T}{(u, v, w)^T}} \hspace{2pt} \bigl( \F^{-1} \vartheta \bigr) (p - \frac{1}{2} t v, q + \frac{1}{2} t u, t)
	\end{align*}
and, since $(p, q, t) \mapsto (p - \frac{1}{2} t v, q + \frac{1}{2} t u, t)$ is measure-preserving, also \eqref{FL1NormTheta}.

If we then use \eqref{EssEstAbove} and $\hut_\st(\gamma) \asymp  \int_{P_\gamma} \hvt_\st(\tilde{\mathcal{P}}) \, d\tilde{\mathcal{P}}$ with constant $C_\st$, we get
	\begin{align*}
		\Norm{\Co_{\widecheck{\pi}_\l^{\mathrm{pr}}} ( L^{\pt, \qt}_{\hvt_\st} )}{f}^{\qt} \leq N_\mathscr{P} \hspace{2pt} C_\qt  \Abs{\l}^{-\frac{(2n+1) \qt}{\pt}} \Norm{\L{1}{\R^{2n+1}}}{\F^{-1} \vartheta}^\qt \sum_{\gamma \in \Gamma} \Norm{\L{\pt}{\R^{2n+1}}}{\F^{-1} \bigl( \vartheta_\gamma \widehat{f} \hspace{2pt} \bigr)}^\qt C_{\st}^\qt \hut_\st(\gamma)^\qt,
	\end{align*}
thus for some $C_2 > 0$ the estimate
	\begin{align*}
		\Norm{\Co_{\widecheck{\pi}_\l^{\mathrm{pr}}} ( L^{\pt, \qt}_{\hvt_\st} )}{f} \leq C_2 \hspace{2pt} \Norm{\ell^\qt_{\hut_\st}}{\Bigl( \Norm{\L{\pt}{\R^d}}{\F^{-1}\bigl( \vartheta_\gamma \hspace{2pt} \widehat{f} \hspace{2pt} \bigr)}\Bigr)_{\gamma \in \Gamma}} = C_2 \hspace{2pt} \Norm{\EE{\pt, \qt}{\st}{n}}{f}.
	\end{align*}

For the estimate from below we notice that for $\gamma \in \Gamma$ and $\mathcal{P} \in \Sigma \HP \gamma$ the function $\Xi \mapsto \vartheta(\Xi \HP \mathcal{P})$ equals $1$ on $\supp(\vartheta_\gamma) \supseteq \Sigma \HP \gamma$. By Young's convolution inequality we get
	\begin{align*}
		\Norm{\L{\pt}{\R^{2n+1}}}{\F^{-1} \bigl( \vartheta_\gamma \widehat{f} \hspace{2pt} \bigr) } \leq C_{\mathscr{P}, \Theta, \pt} \hspace{2pt} \Norm{\L{\pt}{\R^{2n+1}}}{\F^{-1} \bigl( \widehat{f} \hspace{1pt} \pi_\l^{\mathrm{pr}}(0, \mathcal{P}) \vartheta \bigr) }
	\end{align*}
and therefore
	\begin{align*}
		\sum_{\gamma \in \Gamma} \hut_\st(\gamma)^\qt \Norm{\L{\pt}{\R^{2n+1}}}{\F^{-1} \bigl( \vartheta_\gamma \widehat{f} \hspace{2pt} \bigr)}^\qt \leq C^{-\qt}_\st \hspace{2pt} C_{\mathscr{P}, \Theta, \pt}^\qt \Abs{\l}^{\frac{(2n+1) \qt}{\pt}} \Norm{\Co_{\widecheck{\pi}_\l^{\mathrm{pr}}} ( L^{\pt, \qt}_{\hvt_\st} )}{f}^{\qt}.
	\end{align*}
Equivalently, there exists some $C_1 > 0$ such that
	\begin{align*}
		C_1 \hspace{2pt} \Norm{\EE{\pt, \qt}{\st}{n}}{f} = C_1 \hspace{2pt} \Norm{\ell^\qt_{\hut_\st}}{\Bigl( \Norm{\L{\pt}{\R^{2n+1}}}{\F^{-1} \bigl( \vartheta_\gamma \hspace{2pt} \widehat{f} \hspace{2pt} \bigr)}\Bigr)_{\gamma \in \Gamma}} \leq \Norm{\Co_{\widecheck{\pi}_\l^{\mathrm{pr}}} ( L^{\pt, \qt}_{\hvt_\st} )}{f}.
	\end{align*}
Thus, the two norms are equivalent and independent of $\l \in \R \setminus \{ 0 \}$. This proves Case $(1)$.

Case $(2)$ and Case $(3)$ are almost immediate since by Corollary~\ref{CorProjReps} the projective representations $\pi_{{(f_w,f_x,f_y,f_z)}}^{\mathrm{pr}}$ and $\pi_{{(f_u,f_v,f_x,f_y)}}^{\mathrm{pr}}$ coincide with the projective Schr\"{o}dinger representations $\rho_{f_w}^{\mathrm{pr}}$ and $\rho_{\l_{f_x, f_y}}^{\mathrm{pr}}$ with parameters $\l = f_w$ and $\l = \l_{f_x, f_y}$, respectively. Now, it is well known that also for $\l \neq 1$ one has $\Co_{\widecheck{\rho}_\l^{\mathrm{pr}}} \Bigl( \LW{\pt, \qt}{\H/Z(\H)}{\vt_\st} \Bigr) = \M{\pt, \qt}{\vt_\st}{n}$. One either argues as in Case (1) or uses the fact that each $\rho_\l$ can be obtained from $\rho = \rho_1$ by dilation along the central variable $t$.
This proves the Cases $(2)$ and $(3)$, and we are done.
	\end{proof}


The following corollary follows from an application of the embedding theorem for decomposition spaces \cite[Corollary~7.3]{vo16-1}.

	\begin{cor}
For $1 \leq \pt_1 \leq \pt_2 \leq \infty$, $1 \leq \qt_1 \leq \qt_2 \leq \infty$ and $-\infty < \st_2 \leq \st_1 < \infty$, we have
	\begin{align*}
		\E{\pt_1, \qt_1}{\st_1}{n}
		\hookrightarrow
		\E{\pt_2, \qt_2}{\st_2}{n}.
	\end{align*}
	\end{cor}

	\begin{proof}
Let $T_\gamma$ be defined by \eqref{T_gamma} in Lemma \ref{lem_Hmult}.  Since $\det(T_\gamma) = 1$ for all $\gamma \in \Gamma$, $\st_2 - \st_1 \leq 0$, and $\frac{\qt_1}{\qt_2} \leq 1$ implies $(\frac{\qt_1}{\qt_2})' = \infty$ according to Voigtlaender's convention, we get
	\begin{align*}
		K =  \Norm{\ell^{\qt_2 \cdot (\qt_1/\qt_2)'}}{\Bigl( \Abs{\det(T_\gamma)}^{\frac{1}{\pt_1}- \frac{1}{\pt_2}}  \frac{\hut_2(\gamma)}{\hut_1(\gamma)} \Bigr)_{\gamma \in \Gamma}}
		= \Norm{\ell^{\infty}}{\Bigl( (1 + \Abs{\gamma}_{\H}^4)^{\st_2 - \st_1} \Bigr)_{\gamma \in \Gamma}}
		 < \infty
	\end{align*}
and may therefore apply \cite[Corollary~7.3]{vo16-1}.
	\end{proof}

\subsection{The Novelty of $\E{\pt, \qt}{\st}{n}$} \label{Novelty}

In this subsection we prove the second main result of our paper. To be precise, we prove that the spaces $\E{\pt, \qt}{\st}{n}$ form a hitherto unknown class of function spaces on $\R^{2n+1}$, distinct from classical modulation spaces and homogeneous as well as inhomogeneous Besov spaces.

Our proof is built on the classification by Proposition~\ref{MainThm2} and novel results in decomposition space theory by Voigtlaender.
In his thesis~\cite{vo15} and the subsequent preprint~\cite{vo16-1}, Voigtlaender has developed a powerful machinery for comparing decomposition spaces in terms of the geometries of the corresponding underlying coverings $\mathscr{Q}$. While numerous previous papers focused on case-based results (such as partially sharp embeddings between modulation spaces, Besov spaces, and their geometric intermediates, the so-called $\alpha$-modulation spaces; cf.~\cite{ToWa12, SuTo07, Rau06, BoTo05, To04-Conv, To04-Cont_I, To04-Cont_II, Ok04}), Voigtlaender's work provides an extensive collection of necessary and sufficient conditions for embeddings and coincidences of decomposition spaces.

The proof of the following theorem is based on Voigtlaender~\cite{vo16-1}, Theorems~6.9, Theorem~6.9~\sfrac{1}{2}, Lemma~6.10, and Theorem~7.4.

	\begin{thm} \label{MainThm3}
Let $\pt_1, \qt_1, \pt_2, \qt_2 \in [1, \infty]$
and $\st_1, \st_2 \in \R$. Excluding the case when $(\pt_1, \qt_1) = (2, 2) = (\pt_2, \qt_2)$ and $\st_1 = 0 = \st_2$, we have
	\begin{align}
	\E{\pt_1, \qt_1}{\st_1}{n}
	\neq
	\left\{\begin{array}{cl}
	\M{\pt_2, \qt_2}{\vt_{\st_2}}{2n+1}, \\
	\B{\pt_2, \qt_2}{\st_2}{2n+1}, \\
	\hB{\pt_2, \qt_2}{\st_2}{2n+1}.
	\end{array} \right. \label{Non-equalities}
	\end{align}
This holds true in particular if $(\pt_1, \qt_1) \neq (\pt_2, \qt_2)$ or $\st_1 \neq \st_2$.

If $(\pt_1, \qt_1) = (2, 2) = (\pt_2, \qt_2)$ and $\st_1 = 0 = \st_2$, the spaces coincide with $\L{2}{\R^{2n+1}}$.
	\end{thm}

	\begin{proof}
To prove \eqref{Non-equalities}, we proceed by contradiction for all three pairs of spaces from top to bottom. Let us mention that the upper two cases
make use of \cite[Thm.~6.9]{vo16-1} and \cite[Lem.~6.10]{vo16-1}, whereas as the third case makes use of \cite[Thm.~6.9 \sfrac{1}{2}]{vo16-1}. The unweighted $L^2$-case $(\pt_1, \qt_1) = (2, 2) = (\pt_2, \qt_2)$ and $\st_1 = 0 = \st_2$ follows from an application of Pythagoras's theorem or \cite[Lem.~6.10]{vo16-1}.

1) $\E{\pt_1, \qt_1}{\st_1}{n} \neq \M{\pt_2, \qt_2}{\vt_{\st_2}}{2n+1}$: To begin with, we recall that $\M{\pt_2, \qt_2}{\vt_{\st_2}}{2n+1}$ can be characterised as the decomposition space $\DS{Q}{}{\pt_2}{\qt_2}{\ut_{\st_2}}$ of frequency domain $\Orbit_2 = \widehat{\R}^{2n+1}$, covering $\mathscr{Q} := \{ Q_k \}_{k \in (2 \Z)^d}$ with $Q_k := Q + k := (-\e, 2+ \e)^d + k$ and $ \e \in (0, \frac{1}{2})$, and weight $\ut_{\st_2}(k):= (1 + \Abs{k}^2)^{\st_2/2}$. The precise choice of covering will become clear from the proof.

Now, let $(\pt_1, \qt_1) \neq (2, 2)$. If $(\pt_1, \qt_1) \neq (\pt_2, \qt_2)$, the spaces cannot coincide by \cite[Thm.~6.9]{vo16-1}, and we are done. If $(\pt_1, \qt_1) = (\pt_2, \qt_2) \neq (2, 2)$, then by the same theorem they coincide only if the respective frequency coverings which define the spaces are weakly equivalent. In the following we will show that they are not. To be precise, we will show that neither $\mathscr{P}$ is almost subordinate to $\mathscr{Q}$ nor vice versa. Since both coverings are open and connected, it suffices to show 
	\begin{equation*}
	\begin{array}{rclcl}
		N(\mathscr{P}, \mathscr{Q}) &=& \sup_{\gamma \in \Gamma} \Abs{\{ k \in (2 \Z)^{2n+1} \mid P_\gamma \cap Q_k \neq \emptyset \}} 
		&=& \infty \hspace{5pt} \mbox{ and } \\
		N(\mathscr{Q}, \mathscr{P}) &=& \sup_{k \in (2 \Z)^{2n+1}} \Abs{\{ \gamma \in \Gamma \mid P_\gamma \cap Q_k \neq \emptyset \}} &=& \infty.
		\end{array}
	\end{equation*}
To show $N(\mathscr{P}, \mathscr{Q}) = \infty$, we observe that $P = Q$ implies that the set $P$ intersects as many neighbouring $Q_k$ as $Q$ and
	\begin{align*}
		\Abs{\{ k \in \Z \mid P \cap Q_k \neq \emptyset \}} = \Abs{\{ k \in \Z \mid Q \cap Q_k \neq \emptyset \}} = 3^{2n+1}.
	\end{align*}
On the other hand, we consider the cube $[0, 2]^{2n+1} \subset P = Q$ and the edge $\mathscr{E}$ which connects the origin with $\exp(2 X_{p_1}) = (2, 0, \ldots, 0)$. Now, choose $\gamma := (0, b, 0) := \exp(b_1 X_{q_1}) \in \Gamma$, i.e. with $b = (b_1, 0, \ldots, 0) \in (2 \Z)^n$, and consider the paralleleppiped $[0, 2]^{2n+1} \HP \gamma$ and the edge $\mathscr{E} \HP \gamma$. Since the origin is mapped to $\gamma$ and $\exp(2 X_{p_1}) = (2, 0, \ldots, 0)$ is mapped to
$\exp(2 X_{p_1} + b_1 X_{q_1} + 2b_1 X_t) = (2, 0, \ldots, 0, b_1, 0, \ldots,0, 2 b_1)$,
the edge $\mathscr{E} \HP \gamma$ intersects at least $b_1$-many $Q_k$ (in the the $X_t$-direction). As $b_1 \in (2 \Z)$ was arbitrary, this implies $N(\mathscr{P}, \mathscr{Q}) = \infty$.

To show $N(\mathscr{Q}, \mathscr{P}) = \infty$, we observe that the maps $T_\gamma$ from \eqref{T_gamma} in the proof of Proposition~\ref{PropDSDF} are volume-preserving, so the sets $P_\gamma$ become more and more elongated as $\Abs{b_1} \to \infty$. A growing number of sets $P_\gamma$ are therefore required to cover one set $Q + (0, b, 0)$ for large $\Abs{b_1}$, which proves our claim.

We have therefore produced a contradiction to \cite[Thm.~6.9]{vo16-1}, hence the spaces cannot coincide under the assumption $(\pt_1, \qt_1) = (\pt_2, \qt_2) \neq (2, 2)$.

Thus, suppose $\st_1$ or $\st_2$ is different from zero and, without loss of generality, $(\pt_1, \qt_1) = (\pt_2, \qt_2) = (2,2)$. If the two Hilbert spaces coincide, then by \cite[Lem.~6.10]{vo16-1} we have $\hut(\gamma) \asymp \ut(k)$ whenever $P_\gamma \cap Q_k \neq \emptyset$. Since the weights are obviously not equivalent as $\Abs{\gamma} \to \infty$, the spaces cannot coincide. This proves the first case.

2) $\E{\pt_1, \qt_1}{\st_1}{n} \neq \B{\pt_2, \qt_2}{\st_2}{2n+1}$: To begin with, we recall that the inhomogeneous Besov space $\B{\pt_2, \qt_2}{\st_2}{2n+1}$ can be characterised as the decomposition space $\DS{B}{}{\pt_2}{\qt_2}{\dut_{\st_2}}$ of frequency domain $\Orbit_2 = \widehat{\R}^{2n+1}$, covering $\mathscr{B} := \{ B_k \}_{k \in \N_0}$ with
	\begin{align*}
	B_k :=
	\left\{\begin{array}{ccc}
			B_{2^{k+2}}(0) \setminus \overline{B_{2^{k-2}}(0)} &\mbox{ if }& k \in \N, \\
			B_4(0) &\mbox{ if }& k=0,
		\end{array} \right.
	\end{align*}
and weight $\dut_\st := \{ 2^{\st k} \}_{k \in \N_0}$. (See, e.g.~\cite[Def.~9.9]{vo16-1}.)

Our line of arguments is the same as in 1). If $(\pt_1, \qt_1) \neq (2, 2)$ or   $(\pt_2, \qt_2) \neq (2, 2)$, then by the same argument we only need to consider $(\pt_1, \qt_1) = (\pt_2, \qt_2) \neq (2, 2)$ and show that $\mathscr{P}$ and $\mathscr{B}$ are not weakly equivalent.
To be precise, we show that $\mathscr{P}$ is almost subordinate but not weakly equivalent to $\mathscr{B}$. To this end, note that $\sup_{\Xi \in P_\gamma} \Abs{\Xi}$ grows at most linearly in $\gamma$ (for $\gamma$ varying only in the $X_t$-direction it stays constant), while $\sup_{\Xi \in B_k} \Abs{\Xi}$ grows exponentially in $k$. So, while for small $\Abs{\gamma}$ the sets $P_\gamma$ intersect a bounded finite number of $B_k$ and vice versa, for large $\Abs{\gamma}$ we have $\Abs{K_\gamma} := \Abs{\{ k \in \N_0 \mid P_\gamma \cap Q_k \neq \emptyset \}} = 1$ or $2$,
whereas $\Abs{\Gamma_k} := \Abs{\{ \gamma \in \Gamma \mid P_\gamma \cap Q_k \neq \emptyset \}} \to \infty$ as $\Abs{k} \to \infty$. Thus, $N(\mathscr{P}, \mathscr{B}) < \infty$ and $N(\mathscr{B}, \mathscr{P}) = \infty$. Hence, $\mathscr{P}$ is not weakly equivalent to $\mathscr{B}$, but since both coverings are open and connected, $\mathscr{P}$ is almost subordinate $\mathscr{B}$ (cf.~\cite[Prop.~3.6]{fegr85}).
Thus, the spaces cannot coincide by \cite[Thm.~6.9]{vo16-1}.

Thus, if $\st_1$ or $\st_2$ is different from zero and, without loss of generality, $(\pt_1, \qt_1) = (\pt_2, \qt_2) = (2,2)$, then by \cite[Lem.~6.10]{vo16-1} the two Hilbert spaces coincide only if $\hut_{\st_1} \asymp \dut_{\st_2}$. However, since this is never the case, the spaces cannot coincide.

3) $\E{\pt_1, \qt_1}{\st_1}{n} \neq \hB{\pt_2, \qt_2}{\st_2}{2n+1}$: To begin with, we recall that the homogeneous Besov space $\hB{\pt_2, \qt_2}{\st_2}{2n+1}$ can be characterised as the decomposition space $\DS{\dot{B}}{}{\pt_2}{\qt_2}{\dut_{\st_2}}$ of frequency domain $\Orbit_2 = \widehat{\R}^{2n+1} \setminus \{ 0 \}$, covering $\mathscr{\dot{B}} := \{ \dot{B}_k \}_{k \in \Z}$ with $\dot{B}_k := B_{2^{k+2}}(0) \setminus \overline{B_{2^{k-2}}(0)}$, and weight $\hdut_\st := \{ 2^{\st k} \}_{k \in \Z}$. (See, e.g.~\cite[Def.~9.17]{vo16-1}.)
So, the characteristic frequency domain of $\hB{\pt_2, \qt_2}{s_2}{2n+1}$ differs from the frequency domain $\Orbit _1 := \widehat{\R}^{2n+1}$ of $\E{\pt_1, \qt_1}{\st_1}{n}$.

Now, if we suppose that the spaces coincide, then clearly
	\begin{align*}
		\Norm{\E{\pt_1, \qt_1}{\st_1}{n}}{f} \asymp \Norm{\B{\pt_2, \qt_2}{\st_2}{2n+1}}{f}
	\end{align*}
for all $f$ with $\widehat{f} \in C^\infty_c(\Orbit_1 \cap \Orbit_2) = C^\infty_c(\Orbit_2)$. Hence, since $\Orbit_2$ is unbounded, \cite[Thm.~6.9 \sfrac{1}{2}]{vo16-1} implies $(\pt_1, \qt_1) = (\pt_2, \qt_2) = (2,2)$ and $\hut_{\st_1} \asymp \hdut_{\st_2}$. Thus, $(\pt_1, \qt_1) = (2,2) = (\pt_2, \qt_2)$ and $\st_1 = 0 = \st_2$, the negations of both our assumptions, must hold true simultaneously, a contradiction. This completes the proof of the theorem.
	\end{proof}

The following result follows almost for free.

	\begin{cor}
We have the strict embedding
	\begin{align*}
		\E{\pt, \qt}{0}{n} \hookrightarrow \B{\pt, \qt}{\st}{2n+1}
	\end{align*}
for all $\pt, \qt \in [1, \infty]$ and all $\st \leq 0$.
	\end{cor}

	\begin{proof}
Since $\mathscr{P}$ is almost subordinate to $\mathscr{B}$ by the proof of Theorem~\ref{MainThm3}~2), our claim follows from the embedding theorem \cite[Thm.~7.4]{vo16-1} once we show that
	\begin{align*}
		K_\mathbf{r} = \Norm{\ell^{\qt_2 \cdot (\qt_1/\qt_2)'}}{\left( \hdut_\st(k) \hspace{2pt} \Norm{\ell^{\mathbf{r} \cdot (\qt_1/\mathbf{r})'}}{\Bigl( \Abs{\det(T_\gamma)}^{\frac{1}{\pt_1} - \frac{1}{\pt_2}}  / \hut_0(\gamma) \Bigr)_{\gamma \in \Gamma_k}} \right)_{k \in \N_0}} < \infty
	\end{align*}
for $\mathbf{r} := \min \{ \pt, \pt' \}$, $\Gamma_k := \{ \gamma \in \Gamma \mid P_\gamma \cap \dot{B}_k \neq \emptyset \}$, $k \in \N_0$, and $T_\gamma$ defined by \eqref{T_gamma} in the proof of Proposition~\ref{PropDSDF}. Fortunately, $K_\mathbf{r}$ is computed easily because $\Abs{\det(T_\gamma)} = 1$, $\hut_0(\gamma) = 1$, $\hdut_\st(k) = 2^{k \st}$ for $k \in \N_0$, $\Abs{\Gamma_k} \asymp 1$ and $\qt_2 \cdot (\qt_1/\qt_2)' = \qt \cdot 1' = \infty$. In the end, we have
	\begin{align*}
		K_\mathbf{r} \asymp \Norm{\ell^\infty}{\left( 2^{k \st} \right)_{k \in \Z}} < \infty
	\end{align*}
since $\st \leq 0$.
	\end{proof}

\section*{Acknowledgments}

D.R. expresses his gratitude to Karlheinz Gr\"{o}chenig for the excellent working environment at NuHAG and for encouraging him to finish this paper. He also thanks Felix Voigtlaender for interesting discussions and particularly for extending his opus magnum~\cite{vo16-1} by Theorem~6.9~\sfrac{1}{2}. He furthermore thanks Senja Barthel for the visualization of the lattice group $\Gamma$.


David Rottensteiner was supported by the Roth Studentship of the Imperial College Mathematics Department and the Austrian Science Fund (FWF) projects [P~26273 - N25], awarded to Karlheinz Gr\"{o}chenig, [P~27773 - N23], awarded to Maurice de Gosson, and [I~3403], awarded to Stephan Dahlke, Hans Feichtinger and Philipp Grohs.

Michael Ruzhansky was supported by EPSRC grant EP/R003025/1 and by the Leverhulme Grant RPG-2017-151.

\bibliography{Bib_H2n_Paper}
\bibliographystyle{plain}

\end{document}